\documentclass[11pt]{article}
\usepackage[T1]{fontenc}
\usepackage{lmodern}
\usepackage[margin=1in]{geometry}
\usepackage{amsmath,amssymb,amsthm,mathtools,microtype}
\usepackage[hidelinks]{hyperref}
\numberwithin{equation}{section}
\newtheorem{theorem}{Theorem}[section]
\newtheorem{proposition}[theorem]{Proposition}
\newtheorem{lemma}[theorem]{Lemma}
\newtheorem{corollary}[theorem]{Corollary}
\DeclareMathOperator{\grad}{grad}
\newcommand{\dd}{\,d}
\newcommand{\Ch}{\mathrm{Ch}}
\newcommand{\Ric}{\operatorname{Ric}}
\newcommand{\divmu}{\operatorname{div}_{\mu}}
\newcommand{\R}{\mathbb R}
\hypersetup{pdftitle={Continuation and convergence for the higher dimensional conformal Chern-Calabi flow},pdfauthor={}}
\title{Continuation and convergence for the higher dimensional conformal Chern--Calabi flow}

\author{%
Hongyu Chen\thanks{School of Mathematics, Sichuan University,
Chengdu 610064, China.
Email: \texttt{hongyu.chern@gmail.com}.}
\qquad
Li Ma\thanks{Corresponding author.
School of Mathematics and Physics,
University of Science and Technology Beijing,
Beijing 100083, China.
Email: \texttt{nuslma@126.com}.}
\qquad
Li Sheng \thanks{
School of Mathematics, Sichuan University, Chengdu 610064, China.
Email: \texttt{lsheng@scu.edu.cn}.
}}

\date{August 2026}

\begin{document}
\maketitle
\begin{abstract}
We study the volume-normalized gradient flow of the squared Chern scalar
curvature on compact Hermitian manifolds of complex dimension $n\ge3$.
The flow extends smoothly past a finite time if the spatial supremum of
Chern scalar curvature belongs to $L^2$ in time. A uniform curvature bound
in the evolving $L^p$ norm, for one fixed $p>n$, gives global existence and
smooth convergence of the whole trajectory to a constrained critical metric.
On Chern scalar-flat backgrounds, small initial $L^p$ curvature gives
exponential convergence to the background. A corresponding condition on
the trace-free first Chern--Ricci form gives exponential convergence on
Chern--Ricci-flat Hermitian and positive K\"ahler--Einstein backgrounds.
The smallness thresholds are independent of the initial higher derivatives.
We also construct constant Chern scalar curvature backgrounds of either
sign with arbitrarily small smooth perturbations whose solutions, if
global, cannot converge back to the background.
\end{abstract}

\noindent\textit{2020 Mathematics Subject Classification.}
Primary 53E30; Secondary 35K35, 53C55, 35B40.
\smallskip

\noindent\textit{Keywords.} Hermitian manifolds,
Chern scalar curvature, conformal gradient flow, continuation, smooth convergence.

\section{Introduction}

In this work, we study the higher dimensional conformal flow on compact Hermitian manifolds and we extend some of results in \cite{Ma25}.

Let $(M^n,\omega)$ be a smooth compact connected Hermitian manifold without
boundary, of complex dimension $n\ge3$. Normalize
$d\mu=\omega^n/n!$ by $\mu(M)=1$, and write
$$
L=g^{i\bar j}\partial_i\partial_{\bar j},\qquad
R=-g^{i\bar j}\partial_i\partial_{\bar j}\log\det(g_{k\bar l}).
$$
Thus $L$ is nonpositive at a local maximum, and $R$ is the Chern scalar
curvature of $\omega$. The formal adjoint $L^*$ is taken with respect to
the fixed measure $d\mu$. For $\omega_f=e^f\omega$,
\begin{equation}\label{eq:curvature}
R_f=e^{-f}(R-nLf),\qquad d\mu_f=e^{nf}d\mu.
\end{equation}
The squared Chern scalar curvature functional on unit-volume conformal
metrics is
\begin{equation}\label{eq:energy}
E(f)=\int_MR_f^2\dd\mu_f
=\int_Me^{(n-2)f}(R-nLf)^2\dd\mu,\qquad
\int_Me^{nf}\dd\mu=1.
\end{equation}
For the moving inner product
$\langle h,k\rangle_f=\int_Mhk\dd\mu_f$, its constrained negative gradient
flow is
\begin{equation}\label{eq:flow}
f_t=2ne^{-nf}L^*\!\left(e^{(n-2)f}(R-nLf)\right)
-(n-2)e^{-2f}(R-nLf)^2+(n-2)E(f).
\end{equation}

Ma and Sheng obtained existence and convergence results for the conformal
Chern--Calabi flow on compact complex surfaces with Gauduchon metrics;
see also Sheng's lecture announcement~\cite{Sheng25}. 
For the conformal $p$-Calabi functional, Yang and Zhang derive the
Euler--Lagrange equation and construct a minimizing critical metric when
the energy exponent equals the complex dimension~\cite[Theorems~1.1--1.2]{YZ26}. See also Chavel's work in \cite{Chavel70}.
Their variational equation includes the stationary equation of
\eqref{eq:flow}. For the squared functional in dimension $n\ge3$, the
weight $e^{(n-2)f}$ couples the second derivatives to the conformal factor.
The continuation problem requires control of this factor near a finite
terminal time; convergence also requires compactness of the trajectory
as time tends to infinity.

Our first result concerns a finite terminal time.

\begin{theorem}\label{thm:continuation}
Let $0<T<\infty$, and let
$f\in C^\infty(M\times[0,T);\mathbb R)$ solve \eqref{eq:flow}, with
$\int_Me^{nf(0)}\dd\mu=1$. If
\begin{equation}\label{eq:endpoint}
\int_0^T\|R_f(t)\|_\infty^2\dd t<\infty,
\end{equation}
then the solution extends smoothly to $[0,T+\delta)$ for some $\delta>0$.
Its conformal metrics remain positive and have volume one.
\end{theorem}

For the flat unnormalized equation corresponding to \eqref{eq:flow},
the change $f(x,t)\mapsto f(\lambda x,\lambda^4t)$ multiplies scalar
curvature by $\lambda^2$. The integral in \eqref{eq:endpoint} is therefore
invariant under this rescaling. The proof combines weighted derivative bounds with negative moments of
the conformal factor to exclude terminal zeros.

For long-time convergence we use a uniform spatial curvature bound with
one fixed exponent above the dimension threshold.

\begin{theorem}\label{thm:main}
Let $f_0\in C^\infty(M;\mathbb R)$ satisfy
$\int_Me^{nf_0}\dd\mu=1$, and let $f$ be the maximal smooth solution of
\eqref{eq:flow} on $[0,T_{\max})$ with $f(0)=f_0$. If, for some fixed
$p\in(n,\infty)$,
$$
\sup_{0\le t<T_{\max}}
\left(\int_M|R_f(t)|^p\dd\mu_f(t)\right)^{1/p}<\infty,
$$
then $T_{\max}=\infty$. There is a smooth real function $f_\infty$ with
$\int_Me^{nf_\infty}\dd\mu=1$ such that
$$
\lim_{t\to\infty}\|f(t)-f_\infty\|_{C^k(M)}=0
\qquad\text{for every integer }k\ge0.
$$
The limit is a constrained critical point of $E$ and satisfies
\begin{equation}\label{eq:stationary}
-2nL^*\!\left(e^{(n-2)f_\infty}(R-nLf_\infty)\right)
+(n-2)e^{(n-2)f_\infty}(R-nLf_\infty)^2
=(n-2)E(f_\infty)e^{nf_\infty}.
\end{equation}
\end{theorem}

Since the evolving volume is one, a uniform pointwise curvature bound
implies the hypothesis of Theorem~\ref{thm:main}. The limiting equation
is the Euler--Lagrange equation for the squared functional, rather than
the constant Chern scalar curvature equation.

We next give conditions on the initial metric. In the remaining
statements, $f_0$ is smooth and real, and
$\int_Me^{nf_0}\dd\mu=1$. The norms use the indicated evolving metric
and volume form. For real $(1,1)$-forms, we use the pointwise Hermitian
Hilbert--Schmidt norm. The first Chern--Ricci form is
$$
\Ric^{\Ch}(\omega)=-\sqrt{-1}\partial\bar\partial\log\det(g_{k\bar l}).
$$

\begin{theorem}\label{thm:flat-main}
Suppose $R=0$, and fix $p\in(n,\infty)$. There are
$\delta_p,C_p,\beta>0$ and, for every integer $k\ge0$, $C_{p,k}>0$
such that
$$
\|R_{f_0}\|_{L^p(d\mu_{f_0})}\le\delta_p
$$
implies a unique global smooth solution of \eqref{eq:flow} satisfying
\begin{align*}
\sup_{t\ge0}\|f(t)\|_\infty
&\le C_p\|R_{f_0}\|_{L^p(d\mu_{f_0})},\\
\|f(t)\|_{C^k(\omega)}
&\le C_{p,k}\|R_{f_0}\|_{L^p(d\mu_{f_0})}e^{-\beta t}
\qquad(t\ge1).
\end{align*}
The constants $\delta_p,C_p$ depend only on $\omega,p$, and $C_{p,k}$
also depends on $k$. The rate $\beta$ may be chosen depending only on
$\omega$. The threshold and rate are independent of the initial higher
derivatives and of $k$.
\end{theorem}

\begin{corollary}\label{cor:ricci-initial}
Suppose either $\Ric^{\Ch}(\omega)=0$, or that $\omega$ is K\"ahler
and $\Ric^{\Ch}(\omega)=\kappa\omega$ for a constant $\kappa>0$.
Fix $p\in(n,\infty)$. There are $\delta_p,C_p,\beta>0$ and, for every
integer $k\ge0$, $C_{p,k}>0$, with the dependencies in Theorem~\ref{thm:flat-main},
such that
$$
\left\|\Ric^{\Ch}(\omega_{f_0})-\frac{R_{f_0}}n\omega_{f_0}
\right\|_{L^p(\omega_{f_0},d\mu_{f_0})}\le\delta_p
$$
implies a unique global smooth solution of \eqref{eq:flow}, with
\begin{align*}
\sup_{t\ge0}\|f(t)\|_\infty
&\le C_p\left\|\Ric^{\Ch}(\omega_{f_0})
-\frac{R_{f_0}}n\omega_{f_0}\right\|_{L^p(\omega_{f_0},d\mu_{f_0})},\\
\|f(t)\|_{C^k(\omega)}
&\le C_{p,k}\left\|\Ric^{\Ch}(\omega_{f_0})
-\frac{R_{f_0}}n\omega_{f_0}\right\|_{L^p(\omega_{f_0},d\mu_{f_0})}
e^{-\beta t}\quad(t\ge1).
\end{align*}
\end{corollary}

For comparison, Angella, Calamai and Spotti establish existence and
uniqueness up to scaling of constant Chern scalar curvature metrics in
conformal classes of nonpositive Gauduchon degree~\cite{ACS17}, and
study first-Chern--Einstein Hermitian metrics~\cite{ACS20}.
Theorem~\ref{thm:flat-main} and Corollary~\ref{cor:ricci-initial}
give dynamical convergence to the specified background from initial
curvature conditions.

Constant Chern scalar curvature alone does not ensure attraction to the
background. In every complex dimension $n\ge3$, there are unit-volume
K\"ahler backgrounds of positive constant Chern scalar curvature, and also
of negative constant Chern scalar curvature, with smooth normalized
perturbations tending to zero in $C^\infty$ whose solutions cannot converge
back to the background in $C^0$ if they exist for all time
(Proposition~\ref{prop:nonattraction}). These examples have a negative
direction for the constrained second variation. They explain the role of
the positivity assumption in Proposition~\ref{prop:stability}.

A natural question is if there is a finite blow up result for the flow \eqref{eq:flow}. Actually, in another work \cite{CMS26}, we can show the below.
Fix a connected flat complex three-torus $(M,\omega)$ with $\int_Md\mu=1$. In this case, the conformal Chern-Calabi flow \eqref{eq:flow} is
$$
 f_t=-18e^{-3f}L(e^fLf)-9e^{-2f}(Lf)^2
       +9\int_Me^f(Lf)^2d\mu,\qquad \int_Me^{3f}d\mu=1.
$$
Then for every $p\in M$, there exists a real $f_0\in C^\infty(M)$ with $\int_Me^{3f_0}d\mu=1$ such that the maximal positive smooth solution of \eqref{eq:flow} has existence time $0<T_{\max}<\infty$ and
$$
\lim_{t\uparrow T_{\max}}e^{f(t,p)}=0.
$$
In particular it cannot extend across $T_{\max}$ as a positive smooth Hermitian metric. 

Curvature continuation criteria are also known for Calabi flows in
other spaces of metrics. Chen and He prove continuation of the
K\"ahler potential Calabi flow under a uniform Ricci curvature
bound~\cite{ChenHe08}. Shen considers a Mabuchi gradient flow in a
$\partial\bar\partial$-class on Hermitian manifolds with
$c_1^{BC}(M)=0$ and
$\partial\bar\partial\omega=\partial\bar\partial\omega^2=0$.
For that flow, bounded Chern scalar curvature gives continuation, and
a uniform bound for all time gives convergence to the Chern--Ricci-flat
metric in that class~\cite[Theorem~1]{Shen22}. Equation~\eqref{eq:flow}
evolves conformal metrics and is the gradient flow of the
squared curvature functional.

For the main estimates we use the fixed background operator. In the
continuation argument, using rescaling and elliptic estimates outside a set of
times of vanishing measure, we give the weighted derivative bounds.
For convergence, the identity
$|R_f|^n\dd\mu_f=|R-nLf|^n\dd\mu$ gives exponential integrability,
and the $L^p$ bound with $p>n$ gives a uniform H\"older estimate. A common local lifetime and smoothing estimates establish
compactness before the analytic gradient inequality of Feehan and
Maridakis~\cite[Theorem~3]{FM20} is applied. The initial curvature
conditions give a small supremum norm for $f_0$, which enters a local
stability argument. The local construction uses biharmonic estimates
of the type appearing in Koch and Lamm~\cite{KL12}. 

The rest is organized as follows. In Section \ref{sec:variation} we make useful variational identities and elliptic estimates. Local existence result of the flow is stated in Section \ref{sec:local} and proved in  Appendix~\ref{app:local}.
We prove Theorem \ref{thm:continuation} in Section \ref{sec:continuation}.  In Section \ref{sec:convergence}, we prove Theorem \ref{thm:main}. We prove Theorem \ref{thm:flat-main} and Corollary \ref{cor:ricci-initial} in Section \ref{sec:initial}. We prove the required compact-background estimates in Appendix~\ref{app:local}.

\section{Variational identities and elliptic estimates}\label{sec:variation}

Unless otherwise indicated, $C$ denotes a positive constant depending only on the fixed smooth background, and on explicitly fixed exponents or bounds. Its value may change between occurrences. Sobolev and $C^k$ norms without a moving metric are taken with respect to the fixed background. All functions in the static assertions are real and smooth. Let $\gamma$
be the Riemannian metric associated with $\omega$, so that $dV_\gamma=d\mu$
and the principal part of $L$ is $\frac12\Delta_\gamma$. Unmarked distances and covariant derivatives use $\gamma$.
The divergence $\divmu$ and formal adjoints use $d\mu$. On real functions,
$L=\tfrac12\Delta_\gamma+b\cdot\nabla$ for a fixed smooth vector field $b$.
Thus $L\Psi(h)=\Psi'(h)Lh+\tfrac12\Psi''(h)|\nabla h|_\gamma^2$.

\subsection{Variation and the volume density}

Differentiation of~\eqref{eq:energy} gives
$$
DE(f)[h]=(n-2)\int_Mh e^{(n-2)f}(R-nLf)^2d\mu
-2n\int_M e^{(n-2)f}(R-nLf)Lh\,d\mu.
$$
Thus its unrestricted gradient for $\langle\cdot,\cdot\rangle_f$ is
$$
(n-2)e^{-2f}(R-nLf)^2
-2ne^{-nf}L^*\!\left(e^{(n-2)f}(R-nLf)\right).
$$
Since $L1=0$, the weighted mean of this gradient on the volume constraint is $(n-2)E(f)$. Tangent variations have zero $d\mu_f$ mean. Subtracting the gradient's mean and reversing its sign proves~\eqref{eq:flow}. Integration of that equation gives
$$
\frac{d}{dt}\int_Me^{nf}\dd\mu
=n(n-2)E(f)\left(\int_Me^{nf}\dd\mu-1\right).
$$
For a normalized initial value, uniqueness for this scalar ordinary
differential equation preserves the volume constraint. Substitution of
$h=f_t$ into the variation formula then gives
\begin{equation}\label{eq:dissipation}
\frac{d}{dt}\int_Me^{nf}d\mu=0,\qquad
\frac{d}{dt}E(f(t))=-\int_M|f_t|^2d\mu_f.
\end{equation}
We write $\grad E(f)$ for the constrained gradient, so that
$f_t=-\grad E(f)$ along the flow.

\begin{lemma}\label{lem:density}
Let $f$ be a smooth normalized solution of \eqref{eq:flow} on $[0,T)$,
where $T<\infty$. Then $e^{nf(t)}$ has a strong $L^1(d\mu)$ limit of
mass one as $t\uparrow T$. The densities are uniformly integrable on
every interval $[t_0,T)$ with $0\le t_0<T$.
\end{lemma}

\begin{proof}
Normalization and \eqref{eq:dissipation} give
$\|e^{nf(t)/2}\|_2=1$ and
$\|\partial_te^{nf/2}\|_2^2=-\tfrac{n^2}{4}E'(t)$. For $t_0\le s<t<T$,
the time Cauchy--Schwarz inequality therefore gives
\begin{equation}\label{eq:density-cauchy}
\|e^{nf(t)}-e^{nf(s)}\|_1
\le n\sqrt{t-s}\sqrt{E(f(s))-E(f(t))}.
\end{equation}
Here we used that the $L^1$ distance of the densities is at most twice
the $L^2$ distance of their square roots. Since the energy is decreasing
and nonnegative, \eqref{eq:density-cauchy} proves convergence in $L^1$;
the limiting mass is one. The curve with its endpoint attached is compact
in $L^1$. Cover this compact set by finitely many $L^1$ balls of a given
small radius and apply absolute continuity of the integral to their
centers. This proves uniform integrability.
\end{proof}

If $\omega$ is Gauduchon, then $L^*1=0$. In fact
$$
\int_M Lh\,d\mu
=\frac1{(n-1)!}\int_M\sqrt{-1}\partial\bar\partial h\wedge\omega^{n-1},
$$
and two integrations by parts identify vanishing for every $h$ with $\partial\bar\partial\omega^{n-1}=0$. This condition does not require $L$ to be self-adjoint.

\subsection{Fixed elliptic estimates}

\begin{lemma}\label{lem:hessian}
On a fixed compact connected Hermitian manifold of complex dimension $n\ge2$ and volume one, for every $1<p<\infty$,
\begin{equation}\label{eq:hessian-coercivity}
\left\|h-\int_Mh\,d\mu\right\|_{W^{2,p}}
\le C_p\left\|\sqrt{-1}\partial\bar\partial h-\frac{Lh}{n}\omega\right\|_{L^p(\omega,d\mu)}.
\end{equation}
There is a fixed scalar kernel $G(x,y)$ such that
\begin{align}
h(x)-\int_Mh\,d\mu
&=\int_MG(x,y)Lh(y)d\mu(y),\label{eq:scalar-kernel}\\
|G(x,y)|&\le C d_\omega(x,y)^{2-2n}\qquad(x\ne y).\label{eq:kernel-bound}
\end{align}
Also $\|h-\int_Mh\,d\mu\|_{W^{2,p}}\le C_p\|Lh\|_{L^p}$.
\end{lemma}

\begin{proof}
We first show that $\sqrt{-1}\partial\bar\partial h=(Lh/n)\omega$ forces $h$ to be constant. Exterior differentiation and wedging with $\omega^{n-2}$ give
$$
d(Lh)\wedge\omega^{n-1}
+\frac{Lh}{n-1}d(\omega^{n-1})=0.
$$
Let $\theta$ be the Lee form, defined by $d(\omega^{n-1})=\theta\wedge\omega^{n-1}$. Wedging a one-form with $\omega^{n-1}$ is a linear isomorphism, so
$$
d(Lh)=-\frac{Lh}{n-1}\theta.
$$
At a maximum of $h$, $Lh\le0$, and at a minimum, $Lh\ge0$. Connectedness therefore gives a zero of $Lh$. Along any smooth path from this zero the last equation is a homogeneous scalar ordinary differential equation. Its solution with zero initial value is zero. Hence $Lh=0$ everywhere, and the strong maximum principle makes $h$ constant.

For a nonzero real covector $\xi$, the complex Hessian has principal symbol $-\xi^{1,0}\otimes\overline{\xi^{1,0}}$. Its trace-free part has squared norm
$$
\left(1-\frac1n\right)|\xi^{1,0}|_\omega^4>0.
$$
Thus the operator in~\eqref{eq:hessian-coercivity} has injective principal symbol. Its formal adjoint composed with itself is a nonnegative self-adjoint strongly elliptic operator of order four, with kernel the constants. Adding the orthogonal projection onto constants makes it invertible. Composing the inverse on the right with the original second-order adjoint gives a left inverse of order $-2$, whose composition with the trace-free Hessian is exactly $h\mapsto h-\int_Mh\,d\mu$.

The fixed smooth elliptic calculus on a closed manifold gives $L^p\to W^{2,p}$ boundedness for this left inverse; see~\cite[Chapter~7]{T11}. This proves~\eqref{eq:hessian-coercivity}.

The maximum principle also gives $\ker L=\R$. Apply the same construction to $L^*L$ plus the constant projection and compose its inverse with $L^*$. The resulting operator of order $-2$ is a left inverse of $L$ modulo constants. Its kernel gives~\eqref{eq:scalar-kernel} and has the bound~\eqref{eq:kernel-bound}. The smooth off-diagonal part is absorbed into that bound on the compact manifold. Its $L^p\to W^{2,p}$ boundedness proves the last estimate.
\end{proof}

Volume normalization will repeatedly supply the mean of the conformal factor:
\begin{equation}\label{eq:mean-control}
\int_Mf\,d\mu
=-\frac1n\log\int_M\exp\!\left(n\left(f-\int_Mf\,d\mu\right)\right)d\mu,
\qquad
\left|\int_Mf\,d\mu\right|\le\left\|f-\int_Mf\,d\mu\right\|_\infty.
\end{equation}

\subsection{Curvature bounds for the conformal factor}

\begin{lemma}\label{lem:exponential}
Fix $B<\infty$. Every smooth real $f$ satisfying $\int_Me^{nf}\dd\mu=1$ and $\|Lf\|_{L^n(d\mu)}\le B$ obeys, for each fixed $a>0$,
\begin{equation}\label{eq:exponential}
\int_M\exp\left(a\left|f-\int_Mf\dd\mu\right|\right)\dd\mu\le C_{a,B},
\qquad
\int_Me^{af}\dd\mu\le C_{a,B}.
\end{equation}
The constants depend only on the fixed background, $a$ and $B$.
\end{lemma}

\begin{proof}
For $r\ge n$, choose $s$ by $1/s=(n-1)/n+1/r$. On a fixed coordinate ball, the $L^s$ norm of $|x|^{2-2n}$ is at most $Cr^{1-1/n}$: the denominator of its radial integral is $2n-(2n-2)s$, comparable to $r^{-1}$, and $r^{1/s}\le Cr^{1-1/n}$. The kernel representation~\eqref{eq:scalar-kernel}, the bound~\eqref{eq:kernel-bound}, and Young's inequality in finitely many charts therefore give
$$
\left\|f-\int_Mf\dd\mu\right\|_{L^r}\le C B r^{1-1/n}.
$$
Expand the exponential in its nonnegative power series. For integers $j\ge n$, the preceding estimate and $j!\ge(j/e)^j$ bound its $j$th integral term by $(aCBe)^j j^{-j/n}$. This series converges. The finitely many lower moments follow from the $n$th moment, proving the first estimate. The volume constraint and Jensen's inequality give $\int_Mf\dd\mu\le0$, so the second follows from the first.
\end{proof}

We will use a direct change between the two measures. Fix $p>n$. For every nonnegative measurable $F$, H\"older's inequality with exponents $2p/(p+n)$ and $2p/(p-n)$ gives
\begin{equation}\label{eq:holder-transfer}
\|e^fF\|_{L^{(p+n)/2}(d\mu)}
\le \|F\|_{L^p(d\mu_f)}
\left(\int_Me^{(p+n)f}\dd\mu\right)^{(p-n)/(p(p+n))}.
\end{equation}
To see the exponent in the last factor, write
$$
(e^fF)^{(p+n)/2}
=(F^pe^{nf})^{(p+n)/(2p)}e^{(p+n)(p-n)f/(2p)}.
$$

\begin{proposition}\label{prop:elliptic}
Fix $p>n$ and $B<\infty$. For every
$0<\alpha<\min\{1,2-4n/(p+n)\}$, there is a constant $C$, depending only on the fixed background, $p,B$ and $\alpha$, such that every smooth real $f$ satisfying
$$
\int_Me^{nf}\dd\mu=1,\qquad \|R_f\|_{L^p(d\mu_f)}\le B
$$
obeys $\|f\|_{W^{2,(p+n)/2}}+\|f\|_{C^\alpha}\le C$.
\end{proposition}

\begin{proof}
Since $\mu_f(M)=1$, the conformal identity~\eqref{eq:curvature} gives
$$
\|R-nLf\|_{L^n(d\mu)}
=\|R_f\|_{L^n(d\mu_f)}\le B,\qquad
\|Lf\|_{L^n}\le (B+\|R\|_{L^n})/n.
$$
Estimate~\eqref{eq:exponential} bounds $\int_Me^{(p+n)f}\dd\mu$ by a fixed constant. Apply~\eqref{eq:holder-transfer} with $F=|R_f|$. It follows that $\|R-nLf\|_{L^{(p+n)/2}}\le C$, and hence $\|Lf\|_{L^{(p+n)/2}}\le C$.

The scalar elliptic estimate of Lemma~\ref{lem:hessian} and Sobolev embedding give
$$
\left\|f-\int_Mf\dd\mu\right\|_{W^{2,(p+n)/2}}
+\left\|f-\int_Mf\dd\mu\right\|_{C^\alpha}
\le C.
$$
The mean is controlled by~\eqref{eq:mean-control}. This proves the full estimate on the original Hermitian background.
\end{proof}

\section{Local existence and uniform smoothing}\label{sec:local}

The following statement gives a common lifetime from an amplitude bound
and a spatial continuity scale. Appendix~\ref{app:local} contains its proof.

\begin{proposition}\label{prop:local}
For each $B<\infty$, there are $r_0,\delta,c>0$, depending only on $B$ and the fixed background, with the following property. If $g_0\in C^\infty(M,\mathbb R)$ satisfies $\int_Me^{ng_0}\dd\mu=1$ and
$$
\|g_0\|_\infty\le B,\qquad
\sup_{d(x,y)\le r}|g_0(x)-g_0(y)|\le\delta,\qquad 0<r\le r_0,
$$
then \eqref{eq:flow} has a unique smooth solution on $[0,cr^4]$, and
$$
\sup_{0\le t\le cr^4}\|f(t)\|_\infty\le B+1,\qquad
\sup_{cr^4/2\le t\le cr^4}\|\nabla^kf(t)\|_\infty\le C_k r^{-k}\quad(k\ge1).
$$
Here $C_k$ depends only on $B,k$ and the fixed background. All constants
are independent of positive-order derivatives of $g_0$.
\end{proposition}

For a fixed smooth normalized initial function, the oscillation
condition holds when $r$ is sufficiently small. The proposition therefore
gives a unique smooth local solution and a maximal smooth interval.
We will use its uniformity in the initial derivatives both near a finite
terminal time and along an infinite trajectory.

\section{A curvature criterion for continuation}\label{sec:continuation}

Fix a smooth normalized solution $f$ of \eqref{eq:flow} on $[0,T)$,
where $0<T<\infty$, and an earlier time $0<t_0<T$.
Constants in this section may depend on the fixed background, $T,t_0$,
the solution on an earlier compact time interval, its curvature history,
and the uniform integrability modulus in Lemma~\ref{lem:density}.

\subsection{A fixed Green kernel and the inverse conformal factor}
\begin{lemma}
There is a smooth positive function $\pi$, with $L^*\pi=0$ and
$\int_M\pi\dd\mu=1$, and a nonnegative kernel $G_\pi$ such that
\begin{equation}\label{greenrep}
u(x)-\int_Mu\pi\dd\mu
=\int_MG_\pi(x,y)(-Lu(y))\pi(y)\dd\mu(y)
\end{equation}
for smooth $u$, with $\sup_x\int_MG_\pi(x,y)\pi(y)\dd\mu(y)<\infty$.
\end{lemma}

\begin{proof}
The fixed uniformly elliptic operator $L$, with $L1=0$, has a positive
constant-preserving heat semigroup $P_t$ with a smooth strictly positive
kernel. Time averages of transition probabilities have a weak limit on
compact $M$. Translating the averaging interval by a fixed positive time
changes the averages in total variation by a quantity tending to zero,
so the limit is invariant. Applying the adjoint heat operator at time one
gives its smooth strictly positive density $\pi$, and $L^*\pi=0$.

Let $p_t(x,y)$ be the heat kernel relative to $\pi\dd\mu$. Its row mass
and invariant mean are one. Choose $0<\varepsilon<1$ below the minimum
of $p_1$. The kernel $p_1-\varepsilon$ is nonnegative with mass
$1-\varepsilon$, so
$\operatorname{osc}(P_1h)\le(1-\varepsilon)\operatorname{osc}(h)$.
Apply this to $p_1(\cdot,y)$ at successive integer times and use contraction
at the remaining time. Its invariant mean is one, whence
$|p_t(x,y)-1|\le Ce^{-ct}$ for $t\ge1$, uniformly in $x,y$.

The integral $\int_0^\infty(p_t(x,y)-1)\dd t$ has uniformly bounded
$L^1(\pi\dd\mu)$ norm in $y$. It is bounded below: its part over $(0,1)$
is at least $-1$, and its tail is uniformly bounded. Adding a fixed constant
therefore gives a nonnegative kernel $G_\pi$. This addition does not change
\eqref{greenrep}, because $\int_MLu\pi\dd\mu=0$.
Integration of $-\partial_tP_tu$ proves that formula. The row mass identity
for $0<t<1$ and exponential decay for $t\ge1$ give the asserted uniform
kernel bound.
\end{proof}

\begin{lemma}
Along a normalized solution on a finite interval $[t_0,T)$,
\begin{equation}\label{inverseinfty}
\|e^{-f(t)}\|_\infty\le C\bigl(1+\|R_f(t)\|_\infty\bigr).
\end{equation}
\end{lemma}
\begin{proof}
Lemma~\ref{lem:density} gives a fixed $H>1$ such that $\int_{\{e^{nf}>H\}}e^{nf}\,d\mu\le1/4$ at every time. Since $\mu(M)=1$, the set $\{1/4\le e^{nf(t)}\le H\}$ has density mass at least $1/2$, and hence $\mu(\{1/4\le e^{nf(t)}\le H\})\ge1/(2H)$. Positivity of $\pi$ gives $\int_{\{1/4\le e^{nf(t)}\le H\}}\pi\dd\mu\ge\eta>0$ uniformly.

Fix a smooth time. Direct application of the chain rule gives
\begin{align*}
-L\log\bigl(e^{-f}+1+\|R_f\|_\infty\bigr)
={}&\frac{Re^{-f}-R_f}{n(e^{-f}+1+\|R_f\|_\infty)}\\
&-\frac{(1+\|R_f\|_\infty)\tfrac12|\nabla e^{-f}|_\gamma^2}
{e^{-f}(e^{-f}+1+\|R_f\|_\infty)^2}
\le{}&\frac{\|R\|_\infty+1}{n}.
\end{align*}
Formula~\eqref{greenrep} therefore bounds the maximum of the logarithm by its invariant mean plus a fixed constant. On $\{1/4\le e^{nf(t)}\le H\}$ one has $e^{-f}\le4^{1/n}$. If the maximum of $e^{-f}$ exceeds that number, splitting the mean over this set and its complement yields
\begin{align*}
\log\bigl(\max_Me^{-f}+1+\|R_f\|_\infty\bigr)
\le{}&\eta\log\bigl(4^{1/n}+1+\|R_f\|_\infty\bigr)\\
&+(1-\eta)\log\bigl(\max_Me^{-f}+1+\|R_f\|_\infty\bigr)+C.
\end{align*}
Rearrangement proves \eqref{inverseinfty}. If the maximum does not exceed $4^{1/n}$, the same estimate holds directly.
\end{proof}

\subsection{Weighted derivative bounds}
\begin{lemma}\label{closedtop}
Let $L_k$ be smooth real second-order operators, with $L_k1=0$, on balls expanding to $\R^{2n}$. Suppose that their coefficients and smooth positive volume densities converge locally in every spatial $C^m$ norm to those of $\frac12\Delta_{\R^{2n}}$ and a positive constant measure. Let $R_k\to0$ smoothly locally. Suppose that $v_k$ solve
\begin{equation}\label{rescaledflow}
 (v_k)_s=2ne^{-nv_k}L_k^*\!\left(e^{(n-2)v_k}(R_k-nL_kv_k)\right)
 -(n-2)e^{-2v_k}(R_k-nL_kv_k)^2+\lambda_k(s)
\end{equation}
on all fixed spatial balls times $[-a,0]$, where $a>0$, $L_k^*$ is the actual adjoint for the indicated measure, and $\|\lambda_k\|_\infty\to0$. Assume that on every such cylinder the quantities $|v_k|$, $|D^jv_k|$ for $1\leq j\leq4$, and $|(v_k)_s|$ are uniformly bounded. Then a subsequence converges locally, including $s=0$, through fourth spatial derivatives and the first time derivative on $\R^{2n}\times[-a/2,0]$.
\end{lemma}

\begin{proof}
Put $h_k(y,s)=v_k(y,s)-\int_0^s\lambda_k(u)\dd u$. Constant translation of the local expression on the right of \eqref{rescaledflow}, excluding $\lambda_k$, multiplies it by $e^{-2\int_0^s\lambda_k(u)\,du}$. Thus $h_k$ satisfies that local equation with this scalar factor and with no additive forcing. Both $\int_0^s\lambda_k(u)\,du$ and $\lambda_k$ tend uniformly to zero; the scalar factor has a uniform Lipschitz bound.

On a smaller spatial ball, interpolation of $h_k(\cdot,s)-h_k(\cdot,t)$ between its $C^0$ bound $C|s-t|$ and its bounded fourth derivatives gives
$$
 |D^jh_k(x,s)-D^jh_k(x,t)|\leq C|s-t|^{1-j/4},\qquad 1\leq j\leq3.
$$
One may verify this by spatial finite differences of length $|s-t|^{1/4}$; for larger time differences the given bounds suffice. Spatial derivatives through order three are Lipschitz in space. Expansion of the equation therefore gives uniform $C^{\alpha,\alpha/4}$ bounds, for every $0<\alpha<1$, for its principal coefficients and its remaining source. The latter contains derivatives of $h_k$ of order at most three. Let $a_k^{ij}$ be the principal matrix of $L_k$. The positive principal symbol is
$$
 2n^2e^{-2\int_0^s\lambda_k(u)\,du-2h_k}\left(a_k^{ij}\xi_i\xi_j\right)^2.
$$
The assumed amplitude bound gives uniform normal ellipticity on each fixed cylinder.

Extend the principal tensor outside a larger controlled ball by a smooth convex interpolation with a fixed positive quartic symbol. A spatial cutoff supported in that ball has a commutator of order at most three. A time cutoff vanishing near $-a$ and equal to one on $[-a/2,0]$ produces a whole-space linear Cauchy problem with zero initial data. Its coefficients, source, ellipticity, and interval length have uniform bounds.

Apply Amann's Cauchy regularity theorem for uniformly normally elliptic operators of order four on a manifold without boundary \cite[Theorem 1.23(ii) and Remark 1.24]{Amann17}, with differentiability index zero and $0<\alpha'<\alpha<1$. The stronger $\alpha$ bounds give the required little H\"older regularity at exponent $\alpha'$, and $\R^{2n}$ is a uniformly regular manifold without boundary. The estimate bounds the localized solutions in $C^{4+\alpha',1+\alpha'/4}$, uniformly on the closed forward time interval. The constant depends only on the coefficient norms, ellipticity, and the fixed interval length, not on $k$. No times beyond zero are used. Uniqueness identifies the localized smooth solutions with the Cauchy solutions supplied by that theorem. Compact embedding on smaller compact cylinders gives convergence through $D^4h_k$ and $(h_k)_s$. Adding back $\int_0^s\lambda_k(u)\,du$ preserves this convergence. A diagonal subsequence over spatial balls proves the assertion.
\end{proof}

\begin{proposition}\label{weighted}
Suppose that $\int_{t_0}^T\|R_f(t)\|_\infty^2\,dt<\infty$. For every fixed $\sigma>0$ there is a constant $C_\sigma$ such that
\begin{equation}\label{weightedbound}
 e^{\sigma f}\sum_{j=1}^4|\nabla^jf|^{1/j}
       +e^{(\sigma+1/2)f}|f_t|^{1/4}\leq C_\sigma
 \quad\hbox{on }M\times[t_0,T).
\end{equation}
In particular, $f$ has a uniform upper bound on this interval.
\end{proposition}

\begin{proof}
Fix $\sigma>0$. Suppose that the expression on the left of \eqref{weightedbound} is unbounded. Choose levels $P_k\to\infty$ larger than the maximum of the left-hand
side at $t_0$. Let $t_k$ be the first time at which that expression reaches
$P_k$, and choose $x_k$ where the level is attained. The expression is at
most $P_k$ on $M\times[t_0,t_k]$. Smoothness on every earlier compact
interval gives $t_k\to T$. Set
\begin{equation}\label{scales}
  r_k=P_k^{-1}e^{\sigma f(x_k,t_k)},\qquad
 \tau_k=e^{2f(x_k,t_k)}r_k^4=P_k^{-4}e^{4(\sigma+1/2) f(x_k,t_k)}.
\end{equation}
We verify that both scales tend to zero before rescaling. The temporal
summand gives $|\partial_te^{4(\sigma+1/2)f}|\le4(\sigma+1/2)P_k^4$
up to time $t_k$. On every fixed interval $[t_0,t']$, with $t'<T$,
smoothness gives a bound independent of $k$. Integration at $x_k$ yields
$$
e^{4(\sigma+1/2)f(x_k,t_k)}
\le C_{t'}+4(\sigma+1/2)P_k^4(t_k-t')
$$
for large $k$. Divide by $P_k^4$ and send $k\to\infty$ to obtain
$\limsup_k\tau_k\le4(\sigma+1/2)(T-t')$.
Sending $t'\uparrow T$ proves $\tau_k\to0$.

Take $a=1/(8(\sigma+1/2))$. On the physical interval $[t_k-a\tau_k,t_k]$, which lies after $t_0$ for large $k$, the temporal estimate at $x_k$ gives
\begin{equation}\label{anchor}
 \frac12\leq e^{4(\sigma+1/2)[f(x_k,t_k+\tau_ks)-f(x_k,t_k)]}\leq\frac32,
 \qquad -a\leq s\leq0.
\end{equation}
The spatial scale can be checked directly from volume normalization. At time $t_k$, $|\nabla e^{\sigma f}|\le \sigma P_k$, $e^{\sigma f(x_k,t_k)}=P_kr_k$, and $\int_M(e^{\sigma f})^{n/\sigma}d\mu=1$. For a fixed background radius $\rho_0>0$, use the ball of radius $\min(\rho_0,r_k/(2\sigma))$. The ball-volume lower bound and integration along geodesics give
$$
1\ge c\,\min(\rho_0,r_k/(2\sigma))^{2n}(P_kr_k/2)^{n/\sigma}.
$$
If a subsequence had $r_k\ge\varepsilon>0$, the right-hand side would diverge. Thus $r_k\to0$; eventually $\min(\rho_0,r_k/(2\sigma))=r_k/(2\sigma)$, and the same inequality gives $r_k\le C_\sigma P_k^{-1/(2\sigma+1)}$. This argument uses a pointwise integral lower bound and applies also when $n/\sigma<1$.

Choose $\gamma$-orthonormal identifications $T_{x_k}M\simeq\R^{2n}$
and exponential coordinates $\Phi_k(y)=\exp^\gamma_{x_k}(r_ky)$. Every
fixed ball is in their domain for large $k$. Define
$$
 v_k(y,s)=f(\Phi_k(y),t_k+\tau_ks)-f(x_k,t_k).
$$
The rescaled backgrounds are $\gamma_k=r_k^{-2}\Phi_k^*\gamma$, $d\mu_k=r_k^{-2n}\Phi_k^*d\mu$, $L_k=r_k^2\Phi_k^*L$, and $R_k=r_k^2R\circ\Phi_k$. Their coefficients converge smoothly on every fixed ball to the Euclidean backgrounds. The adjoint of $L_k$ relative to $d\mu_k$ is exactly $r_k^2\Phi_k^*L^*$. The equation is \eqref{rescaledflow} with
$$
 \lambda_k(s)=(n-2)\tau_kE(f(t_k+\tau_ks)),\qquad
 \|\lambda_k\|_\infty\leq(n-2)E(f(t_0))\tau_k\longrightarrow0.
$$
The past maximum bound and \eqref{scales} give
\begin{equation}\label{scaledbound}
 e^{\sigma v_k}\sum_{j=1}^4|\nabla_{\gamma_k}^jv_k|^{1/j}
       +e^{(\sigma+1/2) v_k}|(v_k)_s|^{1/4}\leq1,
\end{equation}
with equality at $(0,0)$ and $v_k(0,0)=0$. In particular,
\begin{equation}\label{transformedlip}
 |\nabla_{\gamma_k}e^{\sigma v_k}|\leq \sigma,\qquad
 |\partial_se^{4(\sigma+1/2) v_k}|\leq4(\sigma+1/2).
\end{equation}
The center bounds \eqref{anchor} and the spatial Lipschitz bound imply $v_k\leq C_{\rho,\sigma}$ on $B_\rho\times[-a,0]$ for every fixed $\rho$. This is an upper bound on the entire ball.

The rescaled curvature satisfies
$$
\|e^{-v_k}(R_k-nL_kv_k)\|_{L^\infty(B_\rho)}
\le e^{f(x_k,t_k)}r_k^2\|R_f(t_k+\tau_ks)\|_\infty
=\tau_k^{1/2}\|R_f(t_k+\tau_ks)\|_\infty.
$$
Its squared time integral on $[-a,0]$ tends to zero, because $\tau_k\to0$ and the original squared curvature supremum is integrable. The upper bound for $v_k$, the bounded volume of fixed rescaled balls, and $R_k\to0$ give
\begin{equation}\label{ellipticsmall}
\|L_kv_k\|_{L^2([-a,0];L^{2n}(B_\rho,d\mu_k))}\longrightarrow0
\quad\text{for every fixed }\rho.
\end{equation}

We next obtain the lower bound on the same closed time interval. Fix $\rho>0$ and call $s$ good when $\|L_kv_k(\cdot,s)\|_{L^{2n}(B_{2\rho},d\mu_k)}\leq1$. By \eqref{ellipticsmall}, the Lebesgue measure
$$
 \left|\left\{s\in[-a,0]:
 \|L_kv_k(\cdot,s)\|_{L^{2n}(B_{2\rho},d\mu_k)}>1\right\}\right|
$$
of bad times tends to zero. At a good time solve $L_k\psi_k=L_kv_k$ on $B_{2\rho}$ with zero Dirichlet data. The Dirichlet estimate for the Laplacian on a fixed smooth ball is
$$
 \|\psi\|_{W^{2,2n}(B_{2\rho})}
 \leq C_\rho\|\tfrac12\Delta_{\R^{2n}}\psi\|_{L^{2n}(B_{2\rho})},
 \qquad \psi\in W^{2,2n}\cap W_0^{1,2n};
$$
existence and this estimate follow from \cite[Theorem 9.15 and Lemma 9.17]{GilbargTrudinger01}. The domain is fixed and smooth, $1<2n<\infty$, and its Dirichlet Laplacian has zero kernel. Since $L_k-\tfrac12\Delta_{\R^{2n}}$ tends to zero in the corresponding operator norm, a Neumann series gives the same bound for $L_k$, uniformly for large $k$. The measures $d\mu_k$ are uniformly comparable to Lebesgue measure, and $2n>n$ gives $\|\psi_k\|_\infty\leq C_\rho$.

For a constant $C_{\rho,\sigma}$ large enough, $C_{\rho,\sigma}-v_k+\psi_k$ is positive and $L_k$-harmonic at every good time, uniformly for large $k$. The elliptic Harnack inequality \cite[Corollary 9.25]{GilbargTrudinger01} applies with common ellipticity and drift bounds. Its value at zero is bounded using \eqref{anchor}, so it gives $v_k\geq-C_{\rho,\sigma}$ on $B_\rho$ at every good time. In particular $e^{4(\sigma+1/2) v_k}\geq\delta_\rho>0$ there.

Every time of the closed interval $[-a,0]$ is within twice the displayed measure of a good time for large $k$: otherwise its intersection with a one-sided interval of that radius would contain more bad times than their total measure. This also applies at the endpoints. If that measure is zero, good times can instead be chosen arbitrarily close. The same good time works for all points in $B_\rho$. By the second bound in \eqref{transformedlip}, the change in $e^{4(\sigma+1/2)v_k}$ between these times is at most $8(\sigma+1/2)$ times the displayed measure. For large $k$ this is at most $\delta_\rho/2$, so $e^{4(\sigma+1/2)v_k}\ge\delta_\rho/2$ everywhere on $B_\rho\times[-a,0]$. Thus $v_k$ has two-sided amplitude bounds on every fixed ball throughout the closed interval.

These bounds and \eqref{scaledbound} give all hypotheses of Lemma~\ref{closedtop}. A diagonal subsequence converges through fourth spatial and first time derivatives, including the top, to a function $F$ on $\R^{2n}\times[-a/2,0]$. Equation \eqref{ellipticsmall} gives $\tfrac12\Delta_{\R^{2n}}F=0$, first inside the time interval and then at the endpoints by continuity. Passing to the actual equation \eqref{rescaledflow} yields $F_s=0$, since $R_k\to0$ and $\lambda_k\to0$. The normalization of \eqref{scaledbound} survives at $(0,0)$.

At time zero, \eqref{transformedlip} gives $e^{\sigma F(y)}\leq1+\sigma|y|$ and $F(0)=0$. For a fixed $x$ and large $\rho$ the function
$$
 \sigma^{-1}\log(1+\sigma(|x|+\rho))-F(y)
$$
is nonnegative and harmonic on $B_\rho(x)$. Adding $1$ if necessary makes it positive. The interior gradient estimate for a positive harmonic function, obtained from the Poisson representation and Harnack inequality, bounds $|DF(x)|$ by $C/\rho$ times its value at $x$, which tends to zero as $\rho\to\infty$. Thus $F$ is constant, and $F_s=0$. All terms in the equality in \eqref{scaledbound} vanish, a contradiction. This proves \eqref{weightedbound} for each fixed $\sigma>0$.

To obtain the upper bound for $f$, take $\sigma=1/2$. Then $|\nabla e^{f/2}|\leq C$. Normalization gives $\int_M(e^{f/2})^{2n}\dd\mu=1$, so $e^{f/2}\leq1$ at some point of each time slice. Integrating the spatial Lipschitz bound along a minimizing geodesic gives a uniform global upper bound for $e^{f/2}$ and hence for $f$.
\end{proof}

\subsection{Excluding terminal zeros and restarting}
\begin{lemma}\label{negativemoments}
If $\int_{t_0}^T\|R_f(t)\|_\infty^2\dd t<\infty$, then, for every fixed $s>0$,
\begin{equation}\label{negativebound}
 \sup_{t_0\leq t<T}\int_Me^{-sf(t)}\dd\mu<\infty.
\end{equation}
\end{lemma}

\begin{proof}
Differentiate the integral and insert \eqref{eq:flow}. In its adjoint
term, move $L^*$ onto $e^{-(n+s)f}$. The chain rule gives
$$
L(e^{-(n+s)f})=-(n+s)e^{-(n+s)f}Lf
+\frac{(n+s)^2}{2}e^{-(n+s)f}|\nabla f|_\gamma^2.
$$
Substitute $Lf=(R-(R-nLf))/n$. Combining the resulting square term
with the reaction term in the flow gives
\begin{equation}\label{negativeidentity}
\begin{split}
 \frac{d}{dt}\int_Me^{-sf}\dd\mu
 ={}&-s(n+2+2s)\int_Me^{-(s+2)f}(R-nLf)^2\dd\mu\\
 &-ns(n+s)^2\int_Me^{-(s+2)f}(R-nLf)|\nabla f|_\gamma^2\dd\mu\\
 &+2s(n+s)\int_Me^{-(s+2)f}R(R-nLf)\dd\mu
       -s(n-2)E(f)\int_Me^{-sf}\dd\mu.
\end{split}
\end{equation}
The mixed gradient term is estimated without assigning it a sign. For every $k>0$, integration of the chain rule against the fixed measure yields
$$
 \frac{k^2}{2}\int_Me^{-kf}|\nabla f|_\gamma^2\dd\mu
 =\int_Me^{-kf}L^*1\dd\mu
  +\frac{k}{n}\int_Me^{-kf}(R-(R-nLf))\dd\mu.
$$
Because $L^*1$ and $R$ are fixed bounded functions and $|(R-nLf)|\leq\|R_f\|_\infty e^f$, take $k=s+1$ to obtain
\begin{align*}
 \left|\int_Me^{-(s+2)f}(R-nLf)|\nabla f|_\gamma^2\dd\mu\right|
 &\leq C_s\|R_f\|_\infty
 \left(\|e^{-f}\|_\infty+\|R_f\|_\infty\right)
 \int_Me^{-sf}\dd\mu\\
 &\leq C_s(1+\|R_f\|_\infty^2)\int_Me^{-sf}\dd\mu,
\end{align*}
where \eqref{inverseinfty} was used. The $R(R-nLf)$ integral in \eqref{negativeidentity} has the same bound. The first and last terms of \eqref{negativeidentity} are nonpositive. Gronwall's inequality gives \eqref{negativebound}. In particular, the non-Gauduchon term $L^*1$ has been retained throughout this calculation.
\end{proof}

\begin{proof}[Proof of Theorem~\ref{thm:continuation}]
Fix $0<t_0<T$. Proposition~\ref{weighted} and Lemma~\ref{negativemoments} apply by \eqref{eq:endpoint}.

Use \eqref{weightedbound} with $\sigma=1/2$. The functions $e^{f/2}$ are uniformly Lipschitz in space, and $|\partial_te^{4f}|\leq C$. Thus $e^{4f(t)}$ has a uniform nonnegative continuous limit as $t\uparrow T$, and $e^{f(t)/2}$ converges uniformly to a spatially Lipschitz function $V_T$. Normalization passes to the limit and gives $\int_MV_T^{2n}\dd\mu=1$. By \eqref{negativebound} with $s=n$ and Fatou's lemma,
$$
 \int_MV_T^{-2n}\dd\mu<\infty,
$$
with the integrand assigned the value $+\infty$ at zeros. If $V_T(x_0)=0$, its Lipschitz bound gives $0\leq V_T(x)\leq C d_\gamma(x,x_0)$. Unless $V_T$ is identically zero, which normalization excludes, the preceding inverse integral then dominates a positive multiple of $\int_0^\varepsilon r^{2n-1}r^{-2n}\dd r=\infty$. This contradiction shows $\min_MV_T>0$.

Uniform convergence gives a positive lower bound for $e^{f/2}$ near
$T$. Proposition~\ref{weighted} gives the upper bound. Thus $f$ is bounded
on both sides, and \eqref{weightedbound} with $\sigma=1/2$ bounds
$|\nabla f|_\gamma$ uniformly on a terminal interval.

Choose $B,H<\infty$ so that $\|f(t)\|_\infty\le B$ and
$|f(x,t)-f(y,t)|\le Hd_\gamma(x,y)$ there. Let $r_0,\delta,c$ be the
constants in Proposition~\ref{prop:local} for this $B$, and choose one
$0<r\le r_0$ with $Hr\le\delta$. Restart that proposition at a time
$t_1<T$ so late that $T-t_1<cr^4$. Smooth uniqueness identifies the
restarted solution with the original one on their overlap. It therefore
extends the solution smoothly across $T$ in the original time, with a
positive conformal factor and volume one.
\end{proof}

\section{Global existence and convergence}\label{sec:convergence}

We use the analytic gradient inequality at a smooth constrained critical metric.

\begin{lemma}\label{lem:ls}
Let $f_\infty$ be a smooth normalized solution of \eqref{eq:stationary} and fix $0<\alpha<1$. There are $C<\infty$, $\theta\in[1/2,1)$ and a $C^{4,\alpha}$ neighborhood of $f_\infty$ among normalized functions such that
$$
|E(f)-E(f_\infty)|^\theta\le C\|\grad E(f)\|_{L^2(d\mu_f)}
$$
for every smooth $f$ in that neighborhood. The gradient is taken with respect to $\langle\cdot,\cdot\rangle_f$.
\end{lemma}

\begin{proof}
Use the variable $z=e^{nf/2}$, for which the constraint is $\|z\|_2=1$.
Let $\mathbb S$ be the unit sphere in $L^2(d\mu)$, and let
$\nabla_{\mathbb S}$ denote the gradient for the induced $L^2$ metric.
The energy becomes the real analytic functional
$$
E(\tfrac2n\log z)=\int_Mz^{2(n-2)/n}(R-2L\log z)^2\dd\mu
$$
on the positive open subset of $C^{4,\alpha}(M)$. The moving metric on variations of $f$ is $4/n^2$ times the fixed $L^2$ metric on variations of $z$. Along \eqref{eq:flow}, therefore,
\begin{equation}\label{eq:sphere}
z_t=-\frac{n^2}{4}\nabla_{\mathbb S}E(\tfrac2n\log z),\qquad
\|\nabla_{\mathbb S}E(\tfrac2n\log z)\|_2=\frac2n\|\grad E(f)\|_{L^2(d\mu_f)}.
\end{equation}
Use the sphere chart
$$
\Phi(\eta)=\frac{e^{nf_\infty/2}+\eta}{\|e^{nf_\infty/2}+\eta\|_2},\qquad \eta\perp e^{nf_\infty/2}.
$$
We apply \cite[Theorem 3]{FM20} to $\eta\mapsto E(\tfrac2n\log\Phi(\eta))$. This theorem applies to an analytic gradient map between a Banach domain and its gradient space when its derivatives extend continuously between a weaker domain and a Hilbert gradient space, the embeddings and dual pairings agree, and the two Hessian realizations at the critical point are Fredholm of index zero. In the present case these spaces are
$$
C^{4,\alpha}\cap (e^{nf_\infty/2})^\perp,\quad C^{0,\alpha}\cap (e^{nf_\infty/2})^\perp,\quad
H^4\cap (e^{nf_\infty/2})^\perp,\quad L^2\cap (e^{nf_\infty/2})^\perp.
$$
All pairings and orthogonal complements are in fixed $L^2(d\mu)$. Their embeddings commute, and the induced pairing is definite since $\int_Mu^2\dd\mu>0$ for $u\ne0$.

The chart gradient is analytic from the first space to the second. Its derivative is a fourth-order linear differential operator with coefficients continuous in the $C^{4,\alpha}$ variable, plus smooth multiplier and finite-rank terms from the chart. It therefore extends continuously as $H^4\cap (e^{nf_\infty/2})^\perp\to L^2\cap (e^{nf_\infty/2})^\perp$. At the critical point its Hessian is symmetric, with principal part
$$
8L^*(e^{-2f_\infty}L).
$$
Indeed, the second-order part of the variation of $R-2L\log z$ is $-2z^{-1}L\eta$, so the principal term of the second variation is $8\int_Mz^{-4/n}(L\xi)(L\eta)\dd\mu$.

The full symmetric elliptic Hessian, including the constraint multiplier, is Fredholm of index zero from $H^4$ to $L^2$. Its compression to $(e^{nf_\infty/2})^\perp$, extended by the identity on $\mathbb R e^{nf_\infty/2}$, differs from it by a finite-rank operator. Indeed, symmetry moves the Hessian from a variable argument to the fixed smooth function $e^{nf_\infty/2}$ in each discarded rank-one functional. The compressed operator preserves the orthogonal splitting. Its restriction to the perpendicular subspace is therefore Fredholm of index zero. Elliptic regularity identifies the same smooth kernel and cokernel in $C^{4,\alpha}\to C^{0,\alpha}$, so the index is zero there as well. This verifies the remaining hypotheses of the cited theorem.

It follows that $|E(\tfrac2n\log z)-E(f_\infty)|^\theta$ is bounded by the $L^2$ norm of the chart gradient. The derivative of $\Phi$ and its $L^2$ adjoint are bounded on a smaller chart, so that norm is at most $C\|\nabla_{\mathbb S}E(\tfrac2n\log z)\|_2$. Equation \eqref{eq:sphere} proves the lemma.
\end{proof}

\begin{proof}[Proof of Theorem~\ref{thm:main}]
Proposition~\ref{prop:elliptic} gives uniform constants $B,H$ and $\alpha>0$ such that $\|f(t)\|_\infty\le B$ and $|f(x,t)-f(y,t)|\le Hd(x,y)^\alpha$ for $t<T_{\max}$. Choose one $0<r\le r_0$ with $Hr^\alpha\le\delta$, where $r_0,\delta$ are from Proposition~\ref{prop:local}. The latter gives the same lifetime $cr^4$ from every time slice. Smooth uniqueness identifies each restarted solution with the given one. If $T_{\max}<\infty$, a restart at $t_0>T_{\max}-cr^4/2$ extends it past $T_{\max}$, a contradiction. Hence $T_{\max}=\infty$.

For $t\ge cr^4$ restart at $t-3 cr^4/4$. Its terminal-half estimates imply
\begin{equation}\label{eq:smooth}
\sup_{t\ge cr^4}\|f(t)\|_{C^k}\le C_k\qquad(k\ge0).
\end{equation}
In particular, the tail of the trajectory has compact closure in every $C^k$ topology and $e^f$ has uniform positive lower and upper bounds.

By \eqref{eq:dissipation}, $E(f(t))$ has a nonnegative limit and $\int_0^\infty\|f_t\|_{L^2(d\mu_f)}^2\dd t<\infty$. Choose $t_j\to\infty$ with these squared norms tending to zero. Equation \eqref{eq:smooth} supplies a subsequence converging in every $C^k$ to a smooth normalized $f_\infty$. The flow vector field is continuous under this convergence, and its $L^2$ norm tends to zero because the densities have a common positive lower bound. Thus the limiting vector field vanishes, giving \eqref{eq:stationary} and $\lim_{t\to\infty}E(f(t))=E(f_\infty)$.

In the neighborhood of Lemma~\ref{lem:ls}, if $E(f(t))>E(f_\infty)$, one has
$$
-\frac\dd{\dd t}(E(f(t))-E(f_\infty))^{1-\theta}
=\frac{(1-\theta)\|f_t\|_{L^2(d\mu_f)}^2}{(E(f(t))-E(f_\infty))^\theta}
\ge c\|f_t\|_{L^2(d\mu_f)}.
$$
If the energy equals $E(f_\infty)$ at a finite time, \eqref{eq:dissipation} makes the trajectory stationary thereafter.

It remains to ensure that the trajectory stays in this neighborhood. Put $z(t)=e^{nf(t)/2}$. Its tail has compact closure in $C^{4,\alpha}$ by \eqref{eq:smooth}, and $z(t_j)\to e^{nf_\infty/2}$. There is $\eta_0>0$ such that every point of this compact closure within $L^2$ distance $\eta_0$ of $e^{nf_\infty/2}$ lies in the chart neighborhood for the lemma. Otherwise a sequence violating this property would converge to $e^{nf_\infty/2}$ in $L^2$ and, after extraction, in $C^{4,\alpha}$, a contradiction.

Choose $j$ sufficiently large that
$$
\|z(t_j)-e^{nf_\infty/2}\|_2<\eta_0/4,\qquad
C(E(f(t_j))-E(f_\infty))^{1-\theta}<\eta_0/4.
$$
Here $C$ includes the factor $n/2$ in $\|z_t\|_2=(n/2)\|f_t\|_{L^2(d\mu_f)}$. Integrating the preceding differential inequality up to any time before a first exit from the $L^2$ ball of radius $\eta_0/2$ gives
$$
\int_{t_j}^t\|z_s\|_2\dd s\le C(E(f(t_j))-E(f_\infty))^{1-\theta}<\eta_0/4.
$$
At a first exit the triangle inequality would still give distance strictly less than $\eta_0/2$. Thus no exit occurs and $\int_{t_j}^\infty\|z_t\|_2\dd t<\infty$.

The trajectory converges in $L^2$ to $e^{nf_\infty/2}$. For any $k$, a sequence violating $C^k$ convergence has a smoothly convergent subsequence by \eqref{eq:smooth}, whose $L^2$ limit must be $e^{nf_\infty/2}$, a contradiction. The uniform positive lower bound permits taking logarithms, proving convergence of $f$ in every $C^k$.
\end{proof}

\section{Exponential convergence from small initial curvature}\label{sec:initial}

\subsection{Initial curvature and the conformal factor}

\begin{proposition}\label{prop:scalar-critical}
Suppose $R=0$ and fix $p>n$. There is a constant $C_p$, depending only on the fixed background and $p$, such that every smooth normalized $f$ with $\|R_f\|_{L^p(d\mu_f)}\le1$ satisfies
\begin{equation}\label{eq:flat-critical-entry}
\|f\|_\infty\le C_p\|R_f\|_{L^p(d\mu_f)}.
\end{equation}
\end{proposition}

\begin{proof}
The exact identity $n\|Lf\|_{L^n(d\mu)}=\|R_f\|_{L^n(d\mu_f)}$ gives $\|Lf\|_{L^n}\le1/n$. Lemma~\ref{lem:exponential} therefore bounds $\int_Me^{(p+n)f}\dd\mu$ independently of $f$. Formula~\eqref{eq:holder-transfer}, with $F=|R_f|$, now gives
$$
n\|Lf\|_{L^{(p+n)/2}(d\mu)}
=\|e^fR_f\|_{L^{(p+n)/2}(d\mu)}
\le C_p\|R_f\|_{L^p(d\mu_f)}.
$$
Lemma~\ref{lem:hessian} and $(p+n)/2>n$ bound the centered supremum norm of $f$ by the same right-hand side. Formula~\eqref{eq:mean-control} bounds its mean and proves~\eqref{eq:flat-critical-entry}.
\end{proof}

\begin{proposition}\label{prop:tensor-critical}
Fix $p>n$. There is a constant $C_p$, depending only on the fixed Hermitian background and $p$, such that every smooth normalized $f$ satisfying
\begin{equation}\label{eq:tensor-difference-bound}
\left\|\Ric^{\Ch}(\omega_f)-\frac{R_f}{n}\omega_f
-\Ric^{\Ch}(\omega)+\frac Rn\omega\right\|_{L^p(\omega_f,d\mu_f)}\le1
\end{equation}
obeys
\begin{equation}\label{eq:tensor-critical-entry}
\|f\|_\infty\le C_p
\left\|\Ric^{\Ch}(\omega_f)-\frac{R_f}{n}\omega_f
-\Ric^{\Ch}(\omega)+\frac Rn\omega\right\|_{L^p(\omega_f,d\mu_f)}.
\end{equation}
\end{proposition}

\begin{proof}
The conformal Ricci formula and~\eqref{eq:curvature} give
\begin{equation}\label{eq:tensor-identity}
\Ric^{\Ch}(\omega_f)-\frac{R_f}{n}\omega_f
-\Ric^{\Ch}(\omega)+\frac Rn\omega
=-n\left(\sqrt{-1}\partial\bar\partial f-\frac{Lf}{n}\omega\right).
\end{equation}
For every real $(1,1)$-form $T$, $|T|_{\omega_f}=e^{-f}|T|_\omega$, and therefore
$$
|T|_{\omega_f}^n d\mu_f=|T|_\omega^n d\mu.
$$
Use this identity for the form on the left of~\eqref{eq:tensor-identity}, followed by $\mu_f(M)=1$ and~\eqref{eq:tensor-difference-bound}. Lemma~\ref{lem:hessian} gives a fixed bound for the centered $W^{2,n}$ norm of $f$. In particular $\|Lf\|_{L^n}\le C$, so Lemma~\ref{lem:exponential} gives $\int_Me^{(p+n)f}\dd\mu\le C_p$.

Apply~\eqref{eq:holder-transfer} to the pointwise
$\omega_f$ norm of that same form. By~\eqref{eq:tensor-identity},
$$
n\left\|\sqrt{-1}\partial\bar\partial f-\frac{Lf}{n}\omega\right\|_{L^{(p+n)/2}(\omega,d\mu)}
\le C_p
\left\|\Ric^{\Ch}(\omega_f)-\frac{R_f}{n}\omega_f
-\Ric^{\Ch}(\omega)+\frac Rn\omega\right\|_{L^p(\omega_f,d\mu_f)},
$$
The fixed coercivity estimate~\eqref{eq:hessian-coercivity} and Sobolev embedding bound the centered supremum norm by this right-hand side. Applying~\eqref{eq:mean-control} proves~\eqref{eq:tensor-critical-entry}.
\end{proof}

When $\Ric^{\Ch}(\omega)=(R/n)\omega$, the background difference in~\eqref{eq:tensor-difference-bound} vanishes. This applies to Corollary~\ref{cor:ricci-initial}.

\subsection{Stability near a critical background}

\begin{lemma}\label{lem:abstract-stability}
Let $\mathcal B$ be a fixed self-adjoint elliptic operator of order four on a closed manifold of real dimension $d$, with positive principal symbol, preserving mean-zero functions and having spectrum at least $\eta>0$ on that subspace. Fix $s=4m>d/2+4$, where $m\ge2$ is an integer. Suppose a smooth mean-zero solution of $h_t+\mathcal Bh=N(h)$ satisfies, on a fixed small $H^s$ ball,
$$
\|N(h)\|_{H^{s-2}}\le C\|h\|_{H^s}\|h\|_{H^{s+2}}.
$$
There is a smaller fixed ball, measured by $\|\mathcal B^{s/4}h\|_2$, which the solution cannot exit while smooth. For a solution starting in that ball,
$$
\|\mathcal B^{s/4}h(t)\|_2^2
\le e^{-\eta(t-a)}\|\mathcal B^{s/4}h(a)\|_2^2
$$
for every later smooth time $t$, where $a$ is its starting time.
\end{lemma}

\begin{proof}
All spectral powers in this proof act on mean-zero functions. The equation implies that $N(h(t))$ also has zero mean.
Elliptic estimates for this fixed positive operator give
$$
\|\mathcal B^{s/4}h\|_2\asymp\|h\|_{H^s},\qquad
\|\mathcal B^{(s+2)/4}h\|_2\asymp\|h\|_{H^{s+2}},
$$
and $\|\mathcal B^{(s-2)/4}v\|_2\le C\|v\|_{H^{s-2}}$.
For these orders, the statements follow from the estimates for integer powers and the quadratic form. G{\aa}rding's inequality controls $H^2$ by the form after a lower $L^2$ term; the positive spectral gap controls that lower term. Applying this to the integer powers gives the asserted equivalences. All constants concern one fixed operator.

Self-adjointness and the assumed bound imply
\begin{align*}
\frac{d}{dt}\|\mathcal B^{s/4}h\|_2^2
&=-2\|\mathcal B^{(s+2)/4}h\|_2^2
+2\langle\mathcal B^{(s+2)/4}h,\mathcal B^{(s-2)/4}N(h)\rangle\\
&\le-(2-C\|h\|_{H^s})\|\mathcal B^{(s+2)/4}h\|_2^2.
\end{align*}
Choose the ball so that $C\|h\|_{H^s}\le1$. The positive spectrum gives
$\|\mathcal B^{(s+2)/4}h\|_2^2\ge\eta\|\mathcal B^{s/4}h\|_2^2$.
Integration proves the decay. A first boundary-reaching time would contradict the decrease of this norm.
\end{proof}

\begin{proposition}\label{prop:stability}
Suppose either $R=0$, or that $\omega$ is Gauduchon, $R$ is constant and, for some $\eta>0$,
\begin{equation}\label{eq:gap}
2n^2\|Lh\|_2^2-4n(n-2)R\int_MhLh\,d\mu
-2(n-2)R^2\|h\|_2^2\ge\eta\|h\|_2^2
\end{equation}
for every smooth real $h$ with $\int_Mh\,d\mu=0$. There exist $\delta,C,\beta>0$ and, for every integer $k\ge0$, $C_k>0$, depending only on the fixed background and on $k$ where appropriate, such that every smooth normalized $\|f_0\|_\infty\le\delta$ gives a unique global smooth solution of~\eqref{eq:flow} satisfying
\begin{equation}\label{eq:main-decay}
\sup_{t\ge0}\|f(t)\|_\infty\le C\|f_0\|_\infty,\qquad
\|f(t)\|_{C^k(\omega)}\le C_k\|f_0\|_\infty e^{-\beta t}\quad(t\ge1).
\end{equation}
The threshold and rate are independent of $k$ and of all initial higher derivatives.
\end{proposition}

\begin{proof}[Proof of Proposition~\ref{prop:stability}]
Write $h=f-\int_Mf\,d\mu$. Volume normalization reconstructs
$$
f=h-\frac1n\log\int_Me^{nh}d\mu.
$$
On a fixed small $H^s$ ball with $s>n$, Taylor's formula and $\int_Mh=0$ give
\begin{equation}\label{eq:constant-correction}
\left|\frac1n\log\int_Me^{nh}d\mu\right|\le C\|h\|_2^2,
\qquad \|f\|_{H^s}\le C\|h\|_{H^s}.
\end{equation}
Subtracting the spatial mean from the flow and imposing this reconstruction are equivalent to~\eqref{eq:flow}: the possible difference of reconstructed velocities is constant, and its weighted integral is zero.

For $R=0$, the linearization on mean-zero functions is $-2n^2L^*L$. For constant $R$ on a Gauduchon background it is
\begin{equation}\label{eq:linearization}
-2n^2L^*L+2n(n-2)R(L+L^*)+2(n-2)R^2I.
\end{equation}
To verify the latter, use the normalized path
$f=ah-n^{-1}\log\int_Me^{nah}d\mu$ for mean-zero $h$. Expansion of the energy gives
$$
E(f)=R^2+a^2\left[n^2\|Lh\|_2^2
-2n(n-2)R\int_MhLh\,d\mu-(n-2)R^2\|h\|_2^2\right]+O(a^3).
$$
At a critical point the varying inner product contributes no term to the linearization of the gradient. Polarizing this second variation gives~\eqref{eq:linearization}.

We verify the nonlinear bound required by Lemma~\ref{lem:abstract-stability}. In the scalar-flat case, after subtracting the spatial mean, the remainder consists of
\begin{align*}
&-2n^2(e^{-2f}-1)L^*Lh
-2n^2(n-2)e^{-2f}\langle\nabla h,\nabla Lh\rangle_\gamma\\
&-2n^2(n-2)e^{-2f}\{L^*h-(L^*1)h\}Lh\\
&-n^2(n-2)^2e^{-2f}|\nabla h|_\gamma^2Lh
-n^2(n-2)e^{-2f}(Lh)^2.
\end{align*}
These terms follow by substituting~\eqref{eq:adjoint-exponential} into~\eqref{eq:flow}. The constant part of $f$ occurs only in $e^{-2f}$, because $L$ and $L^*-(L^*1)I$ kill constants.
For the Gauduchon constant-curvature case, add
$$
2n(n-2)R(e^{-2f}-1)(L+L^*)h
+n(n-2)^2R e^{-2f}|\nabla h|_\gamma^2
-(n-2)R^2(e^{-2f}-1+2f),
$$
again subtracting the spatial mean. The last term includes the constant correction in~\eqref{eq:constant-correction}; the linear part has already been removed in~\eqref{eq:linearization}.

Choose $s=4m>n+4$. The space $H^{s-2}$ is a multiplication algebra. The coefficient $e^{-2f}-1$ has $H^{s-2}$ norm at most $C\|h\|_{H^s}$. The four-derivative term is thus bounded by $C\|h\|_{H^s}\|h\|_{H^{s+2}}$. The term with factors of derivative orders one and three is bounded by $C\|h\|_{H^s}\|h\|_{H^{s+1}}$. Products of two second derivatives are bounded by $C\|h\|_{H^s}^2$, and the products of orders two, one and one by $C\|h\|_{H^s}^3$. The additional constant-curvature terms have at most two derivatives and at least two small factors. By~\eqref{eq:constant-correction}, their $H^{s-2}$ norms have the same bound. Subtracting the mean is a bounded operator in these norms. This verifies the required nonlinear estimate.

The fixed operator $2n^2L^*L$ is self-adjoint and elliptic, preserves mean-zero functions and has a positive spectral gap there. Its kernel is the constants, since its quadratic form is $2n^2\|Lh\|_2^2$ and the maximum principle gives $\ker L=\R$. Compactness of its resolvent supplies the gap. In the constant-curvature case, the negative of~\eqref{eq:linearization} is self-adjoint, elliptic, preserves mean-zero functions and has the positive gap~\eqref{eq:gap}. Therefore Lemma~\ref{lem:abstract-stability} applies in each case.

Choose its small energy ball so that the reconstructed $f$ also obeys half the $C^0$ threshold in Proposition~\ref{prop:small-local}. That proposition provides a common time $h_0$ and
$\|f(h_0)\|_{H^s}\le C_s\|f_0\|_\infty$.
Reduce the initial threshold to place the centered $h(h_0)$ strictly inside the chosen ball. The energy argument prevents exit and gives
$$
\|f(t)\|_{H^s}\le C_s\|f_0\|_\infty e^{-\beta(t-h_0)}
$$
at every later smooth time, for a fixed $\beta>0$.

If the maximal smooth interval had a finite endpoint, restart Proposition~\ref{prop:small-local} sufficiently near that endpoint. Its common lifespan and smooth uniqueness extend the original trajectory past the endpoint, a contradiction. Thus the solution is global.
For every $t\ge2h_0$, restart at $t-h_0$ and apply the same proposition:
$$
\|f(t)\|_{C^k}\le C_k\|f(t-h_0)\|_\infty
\le C_k'\|f_0\|_\infty e^{-\beta(t-2h_0)}.
$$
Since $2h_0<1$, this gives~\eqref{eq:main-decay}. The local estimate and the $H^s$ decay give the all-time supremum bound. Only one fixed $s$ was used to choose the initial threshold; it is independent of $k$.
\end{proof}

\begin{proof}[Proof of Theorem~\ref{thm:flat-main}]
Choose the initial curvature threshold at most one. Proposition~\ref{prop:scalar-critical}, for the fixed exponent $p$, bounds $\|f_0\|_\infty$ linearly by the initial norm. Reduce the threshold so that this lies in the fixed neighborhood of Proposition~\ref{prop:stability}. Its scalar-flat case proves global existence and all the stated estimates.
\end{proof}

\begin{proof}[Proof of Corollary~\ref{cor:ricci-initial}]
In both cases the trace-free Chern--Ricci form of the background vanishes.
Proposition~\ref{prop:tensor-critical} therefore bounds $\|f_0\|_\infty$
linearly by the norm in the corollary. If $\Ric^{\Ch}(\omega)=0$, then
$R=0$, and Proposition~\ref{prop:stability} applies.

In the positive K\"ahler--Einstein case, $R=n\kappa>0$. The K\"ahler Bochner identity is
$$
\int_Mg^{i\bar k}g^{j\bar\ell}h_{ij}h_{\bar k\bar\ell}d\mu
=\int_M(Lh)^2d\mu-\int_M\Ric^{i\bar j}h_i h_{\bar j}d\mu.
$$
It follows by commuting one derivative of $\bar\partial h$,
$$
g^{j\bar\ell}\nabla_jh_{\bar k\bar\ell}
=(Lh)_{\bar k}+\Ric_{q\bar k}g^{q\bar a}h_{\bar a},
$$
and integrating against $-g^{i\bar k}h_i$. For an eigenfunction $Lh=-\lambda h$, the identity and $\Ric^{\Ch}(\omega)=\kappa\omega$ imply $\lambda\ge\kappa=R/n$. The left-hand side of~\eqref{eq:gap} on this eigenfunction is
$$
2\big[n^2\lambda^2+2n(n-2)R\lambda-(n-2)R^2\big]\|h\|_2^2.
$$
The polynomial is increasing for $\lambda>0$ and at $\lambda=R/n$ equals $2(n-1)R^2$. Expansion in the orthonormal eigenbasis proves~\eqref{eq:gap}.

We may therefore apply the second part of Proposition~\ref{prop:stability}. In each case reduce the initial curvature
threshold so that the preceding supremum bound enters that proposition.
Its global existence and decay estimates give the stated conclusions.
\end{proof}

\subsection{Failure of attraction to the background}

\begin{proposition}\label{prop:nonattraction}
For each $n\ge3$, there are smooth compact connected K\"ahler manifolds
$(M^n,\omega)$ without boundary and of volume one, with positive constant
Chern scalar curvature, and also such backgrounds with negative constant
Chern scalar curvature, having the following property. On each fixed
background there is a sequence of smooth real functions $f_{0,j}$ such that
$$
\int_Me^{nf_{0,j}}\dd\mu=1,\qquad
f_{0,j}\longrightarrow0\quad\text{in }C^\infty(M).
$$
If the maximal smooth solution $f_j$ of \eqref{eq:flow} with initial value
$f_{0,j}$ exists for all $t\ge0$, then $f_j(t)$ does not converge to zero
in $C^0(M)$ as $t\to\infty$.
\end{proposition}

\begin{proof}
On a K\"ahler background with $R$ constant, one has $L=L^*$ and
$L1=0$. Thus $f=0$ is a stationary solution of \eqref{eq:flow}, with
$E(0)=R^2$.

For $R>0$, take $\mathbb{CP}^1\times T^{n-1}$ with the product of a round
metric and a flat metric. Increasing the scale of the flat factor leaves
$R$ unchanged and makes a positive torus eigenvalue $\lambda$ of $-L$
arbitrarily small. For $R<0$, replace $\mathbb{CP}^1$ by a compact
hyperbolic curve and use the same flat factor. Rescaling the whole product
to volume one multiplies $R$ and $\lambda$ by the same positive factor.
In either case, choose and fix a normalized background for which
$$
n^2\lambda^2+2n(n-2)R\lambda-(n-2)R^2<0.
$$
This is possible because $\lambda/|R|$ can be made arbitrarily small.

Let $h$ be a nonzero real eigenfunction, extended from the torus factor,
with $Lh=-\lambda h$. Since $\int_MLh\dd\mu=0$ and $\lambda>0$,
$\int_Mh\dd\mu=0$. For real $s$ near zero, set
$$
f_s=sh-\frac1n\log\int_Me^{nsh}\dd\mu.
$$
These functions are volume-normalized and tend to zero in $C^\infty$.
The normalized second variation computed in the proof of
Proposition~\ref{prop:stability} gives
$$
E(f_s)=R^2+s^2\bigl[n^2\lambda^2+2n(n-2)R\lambda-(n-2)R^2\bigr]
\|h\|_2^2+O(s^3).
$$
For all sufficiently small nonzero $s$, it follows that $E(f_s)<R^2$.
Choose a sequence $s_j\downarrow0$ in this range and put $f_{0,j}=f_{s_j}$.
The background remains fixed throughout this sequence.

By \eqref{eq:dissipation}, every corresponding smooth solution satisfies
$E(f_j(t))\le E(f_{0,j})<R^2$. On the other hand, for every smooth real $f$,
$$
E(f)\ge e^{-(n-2)\|f\|_\infty}\int_M(R-nLf)^2\dd\mu
=e^{-(n-2)\|f\|_\infty}\bigl(R^2+n^2\|Lf\|_2^2\bigr).
$$
The equality uses $\int_MLf\dd\mu=0$ and $\mu(M)=1$. If a global solution
$f_j(t)$ converged to zero in $C^0$, this lower bound would give
$\liminf_{t\to\infty}E(f_j(t))\ge R^2$, contradicting the strict upper
bound. This proves the assertion.
\end{proof}

\appendix
\section{Quantitative local existence and smoothing}\label{app:local}

The estimates in this appendix prove Proposition~\ref{prop:local} and the small-amplitude statement used in Section~\ref{sec:initial}. Write, on the underlying real manifold,
$$
L=\tfrac12\Delta_\gamma+b\cdot\nabla,
\qquad L^*=\tfrac12\Delta_\gamma-b\cdot\nabla-\divmu b.
$$ Direct expansion gives
\begin{equation}\label{eq:adjoint-identity}
wL^*V=L^*(wV)-\divmu(V\nabla w)+VLw.
\end{equation}
The first-order drift cancels in this identity. We will also use
\begin{equation}\label{eq:adjoint-exponential}
\begin{split}
L^*(e^{(n-2)f}V)=e^{(n-2)f}\big[&L^*V+(n-2)\langle\nabla f,\nabla V\rangle_\gamma\\
&+(n-2)\{L^*f-(L^*1)f\}V+\tfrac12(n-2)^2|\nabla f|_\gamma^2V\big].
\end{split}
\end{equation}
The term $L^*1$ is retained because a scalar-flat background need not be Gauduchon.

\subsection{Linear estimates}

In the Euclidean estimates below, $d\ge1$ is arbitrary and $q>2(d+4)$ is fixed. For the application to $M$, take $d=2n$. For a radius $\rho>0$, we use the cylinders $B_\rho(x)\times(0,\rho^4)$ and $B_\rho(x)\times(\rho^4/2,\rho^4)$. On a fixed compact background, $T$ will be small enough that every radius $\rho<T^{1/4}$ lies in uniformly controlled coordinate charts. On $\R^d$ no restriction on chart radii is needed. Define
\begin{equation}\label{eq:X}
\begin{aligned}
\|u\|_{X_T}={}&\sup_{0<t<T}\big(\|u(t)\|_\infty+t^{1/4}\|\nabla u(t)\|_\infty\big)\\
&+\sup_{x,\,\rho^4<T}\left(\rho^{-d}\int_{B_\rho(x)\times(0,\rho^4)}|\nabla^2u|^2\right)^{1/2}\\
&+\sup_{x,\,\rho^4<T}\left(\rho^{-d}\int_{B_\rho(x)\times(0,\rho^4)}|\nabla u|^4\right)^{1/4}\\
&+\sup_{x,\,\rho^4<T}\rho^{2-(d+4)/q}\|\nabla^2u\|_{L^q(B_\rho(x)\times(\rho^4/2,\rho^4))}.
\end{aligned}
\end{equation}
For $j=0,1,2$, define
\begin{equation}\label{eq:Y}
\begin{aligned}
\|F\|_{Y_{j,T}}={}&
\sup_{x,\,\rho^4<T}\rho^{-d(4-j)/4}\|F\|_{L^{4/(4-j)}(B_\rho(x)\times(0,\rho^4))}\\
&+\sup_{x,\,\rho^4<T}\rho^{4-j-(d+4)(4-j)/(2q)}
\|F\|_{L^{2q/(4-j)}(B_\rho(x)\times(\rho^4/2,\rho^4))}.
\end{aligned}
\end{equation}
These definitions also apply to tensors, using their pointwise norms. All spacetime integrals on $M$ use $d\mu\,dt$. The time supremum in~\eqref{eq:X} is an essential supremum for
distribution solutions. The spaces consist of distributions with finite
displayed norms and are Banach spaces. We keep the source triple
$(F_0,F_1,F_2)$ explicit and equip the direct sum of the source spaces
with the norm $\sum_{j=0}^2\|F_j\|_{Y_{j,T}}$.
For the compact equation below, $F_0$ is scalar, $F_1$ is a one-form,
and $F_2$ is a covariant two-tensor; the indicated adjoints determine
their divergences.

The constants in this section depend on the dimension, on the fixed smooth coefficients and coordinate system in the compact case, and on the indicated exponents. They are independent of $T$ in the Euclidean result and uniform for $T$ in a fixed sufficiently small interval in the compact result.

\begin{lemma}\label{lem:euclidean}
For $0<T<\infty$, the solution of
$$
v_t+\Delta^2v=F_0+\partial_iF_1^i+\partial_i\partial_jF_2^{ij},
\qquad v(0)=v_0,
$$
on $\R^d$ satisfies
\begin{equation}\label{eq:linear-euclidean}
\|v\|_{X_T}\le C\left(\|v_0\|_\infty+
\sum_{j=0}^2\|F_j\|_{Y_{j,T}}\right).
\end{equation}
For bounded $v_0$ and sources in the stated spaces, the biharmonic kernel gives a distribution solution with initial distribution trace $v_0$. The constant is independent of $T$. The same assertion holds for the square of any fixed positive definite constant-coefficient second-order operator, uniformly when its coefficient matrix stays in a fixed compact subset of the positive definite matrices.
\end{lemma}

\begin{proof}
Let $b(x,t)$ be the kernel with spatial Fourier transform $e^{-t|\xi|^4}$. Differentiating its Fourier integral, integrating by parts in $\xi$, and scaling give, for all integers $\ell,m\ge0$,
\begin{equation}\label{eq:kernel}
|D^\ell b(x,t)|\le C_{\ell,m}t^{-(d+\ell)/4}
(1+|x|t^{-1/4})^{-m},\qquad t>0.
\end{equation}
The solution is the convolution of $b$ with $v_0$, plus the time convolutions of $b$, $\partial_i b$, and $\partial_i\partial_jb$ with the three sources. The sign of each differentiated source kernel does not affect the estimates. Under $v(x,t)\mapsto v(\rho x,\rho^4t)$ and $F_j(x,t)\mapsto \rho^{4-j}F_j(\rho x,\rho^4t)$, the norms in~\eqref{eq:X} and~\eqref{eq:Y} are invariant. It therefore suffices to estimate a point at time one and the cylinders of radius one. By linearity, we may assume that the right-hand side of~\eqref{eq:linear-euclidean}, before multiplication by $C$, equals one.

For a point at $(0,1)$, split each source into its portion in $B_2\times(1/2,1)$ and its complement. In the first portion, H\"older's inequality applied to $D^{j+k}b$, for $k=0,1$, is valid whenever
$$
\frac{2q}{4-j}>\frac{d+4}{4-j-k}.
$$
Formula~\eqref{eq:kernel} implies $D^\ell b\in L^a(\R^d\times(0,1))$ for each $a\ge1$ with $a(d+\ell)<d+4$. The source exponents $q/2$, $2q/3$ and $q$ satisfy these inequalities strictly. The near source norms are bounded by a finite spatial cover with cylinders appearing in~\eqref{eq:Y}. For the complement, either time is separated by $1/2$ or spatial distance is bounded below. Formula~\eqref{eq:kernel} then has a summable spatial majorant. A cover by unit balls and the initial-cylinder $L^{4/(4-j)}$ bounds control the contribution. The spatial $L^1$ norms of $b$ and $Db$ control the initial term. Scaling back gives
\begin{equation}\label{eq:pointwise-linear}
\|v\|_{L^\infty(\R^d\times(0,T))}
+\sup_{0<t<T}t^{1/4}\|Dv(t)\|_\infty\le C.
\end{equation}

We next estimate the two initial-cylinder integrals. Choose $0\le\chi\le1$, supported in $B_2$ and equal to one on $B_1$. Twice integrating by parts in the Hessian components gives
$$
\int_0^1\int\chi^4|D^2v|^2
\le2\int_0^1\int\chi^4(\Delta v)^2
+C\int_0^1\int_{B_2}|Dv|^2.
$$
The Hessian cross terms can be absorbed because $D(\chi^4)$ has a factor $\chi^3$ and $D^2(\chi^4)$ has a factor $\chi^2$. Pair the equation with $\chi^4v$ and integrate to time one. Discarding the nonnegative final-time term, the cutoff contributions are bounded by $C\int_{B_2\times(0,1)}(|Dv|^2+|v|^2)$ and a small multiple of the Hessian integral. The three source pairings are bounded by
\begin{align*}
&\|F_0\|_{L^1(B_2\times(0,1))}\|v\|_\infty,\\
&C\|F_1\|_{L^{4/3}(B_2\times(0,1))}
\left(\left(\int_0^1\int\chi^4|Dv|^4\right)^{1/4}+\|v\|_\infty\right),\\
&C\|F_2\|_{L^2(B_2\times(0,1))}
\left(\left(\int_0^1\int\chi^4|D^2v|^2\right)^{1/2}
+\|Dv\|_{L^2(B_2\times(0,1))}+\|v\|_\infty\right).
\end{align*}
A finite cover bounds these source norms. The cutoff weights follow from $\chi^{16}\le\chi^4$ and $\chi^8\le\chi^4$. Equation~\eqref{eq:pointwise-linear} gives $\int_{B_2\times(0,1)}|Dv|^2\le C\int_0^1t^{-1/2}dt\le C$.

Integrating $\operatorname{div}(\chi^4|Dv|^2Dv)$ against $v$, and using Young's inequality, gives
$$
\int_0^1\int\chi^4|Dv|^4
\le C\|v\|_\infty^2\int_0^1\int\chi^4|D^2v|^2+C\|v\|_\infty^4.
$$
Here $\|v\|_\infty\le C$ by~\eqref{eq:pointwise-linear}. The preceding energy pairing therefore yields
$$
\int_0^1\int\chi^4|D^2v|^2
\le C\left[1+
\left(\int_0^1\int\chi^4|D^2v|^2\right)^{1/2}
+\left(\int_0^1\int\chi^4|D^2v|^2\right)^{1/4}\right].
$$
Absorption bounds both integrals. Scaling proves the initial-cylinder terms of~\eqref{eq:X}.

It remains to control $D^2v$ in $L^q(B_1(0)\times(1/2,1))$. Separate sources in $B_2\times(1/4,1)$ from their complement. The far contributions and the initial term are controlled by~\eqref{eq:kernel}. The near sources have bounded $L^{2q/(4-j)}$ norms: the intervals $(1/4,1/2)$ and $(1/2,1)$ are the terminal intervals for radii $2^{-1/4}$ and one, respectively, and finite spatial covers suffice. For $F_0\in L^{q/2}$, Young's convolution inequality uses $D^2b\in L^{q/(q-1)}$; for $F_1\in L^{2q/3}$, it uses $D^3b\in L^{2q/(2q-1)}$. Both follow from $q>(d+4)/2$.

For a tensor $F_2\in L^q$, each component of the required spacetime operator has Fourier multiplier
$$
\frac{\xi_a\xi_b\xi_i\xi_j}{\sqrt{-1}\sigma+|\xi|^4}.
$$
It is bounded on $L^2$ by Plancherel's theorem. Away from the spacetime origin its kernel is $\mathbf1_{t>0}\partial_a\partial_b\partial_i\partial_jb(x,t)$. Formula~\eqref{eq:kernel} and its time derivative bound the kernel, its spatial gradient and its time derivative, respectively, by $C\max\{|x|,|t|^{1/4}\}^{-d-4}$, $C\max\{|x|,|t|^{1/4}\}^{-d-5}$ and $C\max\{|x|,|t|^{1/4}\}^{-d-8}$. At $x\ne0$ these quantities tend to zero as $t\downarrow0$, so the extension by zero has no off-diagonal jump. The bounds imply the integral H\"ormander smoothness condition in the parabolic metric. Metric balls have measure comparable to $\rho^{d+4}$ and are doubling. The Calder\'on--Zygmund theorem for an $L^2$-bounded singular integral on this homogeneous metric space therefore gives $L^q$ boundedness for $1<q<\infty$; see~\cite[Chapter~III]{CW71}. This proves the terminal-cylinder estimate.

For sources in the stated spaces, the kernel integrals converge by the pointwise estimates. The equation and the integral estimates can be obtained by first truncating the sources away from time zero and regularizing, and then passing to distributions and using weak lower semicontinuity on each fixed cylinder. One may choose the regularization scale smaller than the truncation scale; enlarged or shifted cylinders are covered by a fixed number of comparable admissible cylinders. Thus the source bounds remain uniform. The initial trace is the distribution trace of the homogeneous kernel term; the zero-initial-data terms have zero trace. Uniqueness for the constant-coefficient Cauchy problem in this bounded distribution class follows by testing against the explicit backward kernel. A constant linear change of variables proves the final assertion.
\end{proof}

\begin{lemma}\label{lem:compact-linear}
Let $M$ be a closed smooth real $d$-manifold with fixed smooth positive measure and background Riemannian metric. Let $L$ be a real smooth uniformly elliptic second-order scalar operator with $L1=0$. There exist $T_*>0$ and $C>0$ such that, for $0<T\le T_*$, the distribution solution of
$$
u_t+L^*Lu=F_0+\nabla^*F_1+(\nabla^2)^*F_2,
\qquad u(0)=u_0,
$$
has initial distribution trace $u_0$, is unique in $X_T$, and satisfies
\begin{equation}\label{eq:compact-linear}
\|u\|_{X_T}\le C\left(\|u_0\|_\infty+\sum_{j=0}^2\|F_j\|_{Y_{j,T}}\right).
\end{equation}
The signs chosen for the divergence adjoints do not affect the estimate. For smooth data the constructed solution agrees with the classical solution.
\end{lemma}

\begin{proof}
For a bounded coefficient $c$ and $0\le j,k\le2$ with $j+k\le3$, the definitions and H\"older's inequality give
\begin{equation}\label{eq:lower}
\|cD^ku\|_{Y_{j,T}}\le C\|c\|_\infty
T^{(4-j-k)/4}\|u\|_{X_T}.
\end{equation}
For $j=k=2$ the same inequality holds without a time factor. On $B_\rho\times(0,\rho^4)$, use the $L^\infty$ bound for $k=0$, the $L^4$ bound for $k=1$, and the $L^2$ bound for $k=2$. The resulting power is $\rho^{4-j-k}$ in all three cases. On $B_\rho\times(\rho^4/2,\rho^4)$, use the pointwise gradient bound or the terminal $L^q$ Hessian bound, with $2q/(4-j)\le q$, and obtain the same power. Coordinate Hessians differ from intrinsic Hessians by a bounded coefficient times $Du$, which has a positive time gain. Likewise, if $0\le\ell\le j\le2$,
\begin{equation}\label{eq:source-lowering}
\|cF_j\|_{Y_{\ell,T}}\le C\|c\|_\infty
T^{(j-\ell)/4}\|F_j\|_{Y_{j,T}}.
\end{equation}

Write $A^{ij}$ for the principal coefficients of the operator $L$
in this lemma. In coordinates, the fourth-order principal part of $L^*L$ is
$$
\partial_i\partial_j(A^{ij}A^{k\ell}\partial_k\partial_\ell u).
$$
Every lower-order term can be written as $D^j(cD^ku)$ with $0\le j,k\le2$ and $j+k\le3$. In particular, a third derivative is put into first-divergence form acting on $D^2u$. Smooth volume densities only change the lower-order coefficients.

Choose coordinate balls of radius $\rho_0$ with uniformly bounded overlap, a partition of unity $\chi_\alpha$, and slightly larger cutoffs $\eta_\alpha$ equal to one near $\operatorname{supp}\chi_\alpha$. Their fixed-size enlargements have bounded overlap as well. Arrange $T^{1/4}\le c\rho_0$. Freeze the principal coefficients at the center of each chart. The frozen operator is $(A_\alpha^{ij}\partial_i\partial_j)^2$, so Lemma~\ref{lem:euclidean} applies uniformly.

For a source triple $F=(F_0,F_1,F_2)$, multiplying its associated
distribution by $\chi_\alpha$ and expanding the adjoints gives an exact local source triple. Terms that retain $j$ divergences are bounded in $Y_{j,T}$; those with fewer divergences are bounded using~\eqref{eq:source-lowering}. Extend these supported source coefficients by zero in the coordinate space. Solve the frozen zero-data equations and multiply their solutions by $\eta_\alpha$ before summing globally. This defines a linear map $\mathcal R$ from source triples to $X_T$. Define $\mathcal R_0u_0$ by the same construction with the frozen homogeneous solutions of initial data $\chi_\alpha u_0$. Their initial sum is $u_0$, because $\eta_\alpha\chi_\alpha=\chi_\alpha$.

For each local solution $v$, the only fourth-order residual is
$$
\partial_i\partial_j\big(\eta_\alpha(A^{ij}A^{k\ell}-A_\alpha^{ij}A_\alpha^{k\ell})
\partial_k\partial_\ell v\big).
$$
All commutators have the form in~\eqref{eq:lower}. For example, $(D\eta_\alpha)D^3v$ is rewritten as a first divergence of $(D\eta_\alpha)D^2v$ and a coefficient times $D^2v$. Passing back to intrinsic divergences introduces only terms covered by~\eqref{eq:source-lowering}. Consequently there are concrete operators $\mathcal K$ on the source direct sum and $\mathcal K_0$ on initial data such that
\begin{align*}
(\partial_t+L^*L)\mathcal RF
&=F_0+\nabla^*F_1+(\nabla^2)^*F_2
+(\mathcal KF)_0+\nabla^*(\mathcal KF)_1+(\nabla^2)^*(\mathcal KF)_2,\\
(\partial_t+L^*L)\mathcal R_0u_0
&=(\mathcal K_0u_0)_0+\nabla^*(\mathcal K_0u_0)_1
+(\nabla^2)^*(\mathcal K_0u_0)_2.
\end{align*}
If the oscillation of $A^{ij}A^{k\ell}$ in the enlarged charts is at most $\epsilon$, then
$$
\|\mathcal K\|\le C\epsilon+C_{\rho_0}T^{1/4},\qquad
\|\mathcal K_0u_0\|\le(C\epsilon+C_{\rho_0}T^{1/4})\|u_0\|_\infty.
$$
The constant in front of $\epsilon$ does not depend on the total number of charts: a ball of radius at most $T^{1/4}$ meets only uniformly many cutoff supports. Cutoff derivative constants are included in $C_{\rho_0}$. Choose $\rho_0$ first so that $C\epsilon<1/4$, and then choose $T_*$ so that $C_{\rho_0}T_*^{1/4}<1/4$. The inverse $(I+\mathcal K)^{-1}$ on the source direct sum exists by its Neumann series. The function
$$
u=\mathcal R_0u_0+\mathcal R(I+\mathcal K)^{-1}(F-\mathcal K_0u_0)
$$
solves the required equation and satisfies~\eqref{eq:compact-linear}.

For uniqueness, localize a zero-source, zero-initial-trace solution in $X_T$ and apply the frozen estimate to its coefficient errors and commutators. Bounded overlap gives
$$
\|u\|_{X_T}\le(C\epsilon+C_{\rho_0}T^{1/4})\|u\|_{X_T}.
$$
After a further fixed reduction of $\rho_0$ and $T_*$, the coefficient is less than one, so $u=0$. The localized frozen representation is legitimate by constant-coefficient distribution uniqueness. This proves independence of the source representation and agreement with a classical solution.
\end{proof}

The same localization argument applies to $L^*(e^{-2h}L)$ for a fixed smooth function $h$. More precisely, suppose $\|h\|_\infty\le B$ and $\|\nabla^jh\|_\infty\le C_j\delta r^{-j}$ for $j\ge1$. On radius-$r$ charts the principal coefficient has oscillation at most $C_B(\delta+r)$. The coefficients of total derivative order $\ell\le3$ are bounded by $C_Br^{-(4-\ell)}$. On $T=\vartheta r^4$, formula~\eqref{eq:lower} bounds their residual by $C_B\vartheta^{(4-\ell)/4}$. Choosing $\delta,r$ small and then $\vartheta$ small makes the Neumann-series residual less than $1/2$, exactly as in the preceding proof. Thus
\begin{equation}\label{eq:linear}
\|u\|_{X_T}\le C_B\left(\|u(0)\|_\infty+\sum_{j=0}^2\|F_j\|_{Y_{j,T}}\right)
\end{equation}
for $u_\tau+L^*(e^{-2h}Lu)=F_0+\divmu F_1+L^*F_2$. The latter second-order source is converted to coordinate second divergences and lower-order terms by the same localization. All constants are uniform over the functions $h$ satisfying these bounds.

On an initial cylinder the $L^2$ norm of $D^2u$ is at most $C\rho^{d/2}\|u\|_{X_T}$ and the $L^4$ norm of $Du$ is at most $C\rho^{d/4}\|u\|_{X_T}$. On a terminal cylinder their $L^q$ and $L^{2q}$ norms are bounded by
$$
C\rho^{-2+(d+4)/q}\|u\|_{X_T},\qquad
C\rho^{-1+(d+4)/(2q)}\|u\|_{X_T}.
$$
Consequently the products $uD^2u$, $DuD^2u$, $(D^2u)^2$ and $D^2u|Du|^2$ lie in $Y_{2,T},Y_{1,T},Y_{0,T},Y_{0,T}$, respectively, with quadratic or cubic norm bounds. On the initial cylinders this uses $L^\infty\cdot L^2$, $L^4\cdot L^2$, $L^2\cdot L^2$ and $L^2\cdot L^4\cdot L^4$; on terminal halves it uses $L^\infty\cdot L^q$, $L^{2q}\cdot L^q$, $L^q\cdot L^q$ and $L^q\cdot L^{2q}\cdot L^{2q}$. The resulting powers of $\rho$ are those in the definitions of the source spaces. Factoring one difference at a time gives locally Lipschitz bounds, and repeated differentiation gives smooth source maps on a fixed $X_T$ ball. Smooth scalar functions of $u$ act through its $L^\infty$ component.

\subsection{A common lifetime from a continuity scale}

\begin{proof}[Proof of Proposition~\ref{prop:local}]
Omit for the moment the last, spatially constant term in \eqref{eq:flow}, and rescale its time $s$ by $\tau=2n^2s$. The unnormalized solution $g$ then satisfies
\begin{equation}\label{eq:divergence}
\begin{split}
g_\tau={}&-L^*(e^{-2g}Lg)-n\divmu(e^{-2g}Lg\nabla g)\\
&+e^{-2g}\left[\frac{n+2}{2}(Lg)^2-\frac{n^2}{2}Lg|\nabla g|_\gamma^2\right]\\
&+\frac1nL^*(e^{-2g}R)+\divmu(e^{-2g}R\nabla g)\\
&+e^{-2g}\left[-\frac2nRLg+\frac{nR}{2}|\nabla g|_\gamma^2-\frac{n-2}{2n^2}R^2\right].
\end{split}
\end{equation}
For example, this follows by applying the product identity with $w=e^{-ng}$ and $V=e^{(n-2)g}(R-nLg)$ and using $L(e^{-ng})=e^{-ng}(-nLg+\tfrac{n^2}{2}|\nabla g|_\gamma^2)$.

Average $g_0$ by a smooth nonnegative normalized kernel supported at scale $r/8$, and denote the resulting smooth function by $h$. Subtracting the constant $g_0(x)$ when differentiating the normalized kernel gives
\begin{equation}\label{eq:average}
\|h\|_\infty\le B,\qquad \|h-g_0\|_\infty\le\delta,\qquad
\|\nabla^jh\|_\infty\le C_j\delta r^{-j}\quad(j\ge1).
\end{equation}
Write $g=h+u$ and work on $0<\tau<T=\vartheta r^4$. Estimate~\eqref{eq:linear}, in real dimension $d=2n$, applies to $u_\tau+L^*(e^{-2h}Lu)$ and its three source terms. The choices of $\delta,r_0,\vartheta$ make its inverse uniform.

We detail the nonlinear bounds. In the second-divergence source, the principal difference is
$$
(e^{-2h}-e^{-2(h+u)})Lu+(e^{-2h}-e^{-2(h+u)})Lh.
$$
The first term is bounded by $C_B\|u\|_\infty\|u\|_{X_T}$ in $Y_{2,T}$. The second contributes at most $C_B\delta\vartheta^{1/2}\|u\|_{X_T}$. The other highest-order products are $Lu\nabla u$, $(Lu)^2$ and $\frac12Lu|\nabla u|_\gamma^2$. On initial cylinders they are controlled respectively by $L^2\cdot L^4$, $L^2\cdot L^2$ and $L^2\cdot L^4\cdot L^4$; on terminal halves use $D^2u\in L^q$ and $Du\in L^{2q}$ in the preceding subsection. Their norms are quadratic or cubic in $\|u\|_{X_T}$. Each factor $\nabla^jh$ carries $\delta r^{-j}$, canceled by the corresponding power of $T^{1/4}$ from \eqref{eq:lower}.

The curvature terms in \eqref{eq:divergence} have total derivative order at most two. Their differences have a factor $T^{1/2}$ when both derivatives fall on $u$, a factor $\delta r^{-1}T^{3/4}$ when one falls on $h$, or a factor $\delta r^{-2}T$ when both fall on $h$; the term with $R^2$ has factor $T$. Thus on an $X_T$ ball of radius $\varepsilon\le1$ the difference of the nonlinear source tuples is bounded by
$$
C_B(\varepsilon+\delta\vartheta^{1/4}+T^{1/2})\|u-v\|_{X_T}.
$$
The right-hand side of \eqref{eq:divergence} at $g=h$ is a smooth source with $Y_{0,T}$ norm at most $C_BT(\delta r^{-4}+1)$. Choose $\varepsilon$ small first, and then shrink $\delta,\vartheta,r_0$ so that this source and $u(0)=g_0-h$ map into half the ball and the Lipschitz constant in \eqref{eq:linear} is less than $1/2$. The contraction mapping theorem gives a unique solution with $\|g-g_0\|_\infty\le1/8$ after a further reduction of these constants.

Here are the positive-time smoothing details. Rescale space by $r$ and time by $r^4$ in a smaller chart. All derivatives of $h$ are bounded by \eqref{eq:average}. Pull back $g$ by the flows of compactly supported coordinate vector fields, using a smooth time cutoff that vanishes near zero and equals one on a fixed closed interior time interval. In the smaller chart these flows are translations on that interval. The transformed principal coefficients change by $O(|a|)$ for translation parameter $a$; the other new terms have total order at most three, including the drift from the time cutoff. The linear and product estimates above bound these perturbations on $X_T$ uniformly for small $a$. Pullback is bounded on $X_T$: cylinders near zero are unchanged, and affected cylinders can be covered by finitely many comparable cylinders away from zero. The Hessian chain rule adds only a bounded coefficient times $Du$, controlled by \eqref{eq:lower}.

The transformed equation is a smooth map of $a$ and $u$ into the source spaces. Its derivative in $u$ is invertible by the same contraction bound. The smooth implicit function theorem therefore bounds every parameter derivative in $X_T$. On the smaller chart these are the spatial distribution derivatives of $g$, whose suprema are bounded by the $L^\infty$ component of $X_T$. The equation then gives all time derivatives. Spatial mollification and integration in time give continuous representatives of these derivatives, proving smoothness and the bounds $\|\nabla^kg\|_\infty\le C_kr^{-k}$ on closed interior time intervals. This argument uses only smooth background coefficients. For smooth initial data, the unnormalized equation has a unique classical solution on a short interval by~\cite[Theorem~1.1]{MM12}: its principal symbol is $-\tfrac14e^{-2g}(\gamma^{ij}\xi_i\xi_j)^2$, a negative product of two positive quadratic forms, and its lower coefficients depend smoothly on derivatives through order three. This solution lies in the contraction ball on a sufficiently short interval: the derivative components tend to zero as that interval shrinks. Uniqueness identifies the two solutions and supplies smoothness at the initial time.

Return to time $s$. Integration of the unnormalized equation gives
$$
\frac{d}{ds}\int_Me^{ng(2n^2s)}\dd\mu=-n(n-2)E(g(2n^2s)).
$$
Define
\begin{equation}\label{eq:normalize}
\begin{split}
t(s)&=\int_0^s\left(\int_Me^{ng(2n^2\sigma)}\dd\mu\right)^{-2/n}\dd\sigma,\\
f(t(s))&=g(2n^2s)-\frac1n\log\int_Me^{ng(2n^2s)}\dd\mu.
\end{split}
\end{equation}
Adding a constant $a$ multiplies the unnormalized vector field by $e^{-2a}$
and the energy by $e^{(n-2)a}$. Also
$$
\frac{\dd s}{\dd t}=\left(\int_Me^{ng(2n^2s)}\dd\mu\right)^{2/n},\qquad
-\frac1n\frac{d}{ds}\log\int_Me^{ng(2n^2s)}\dd\mu
=\frac{(n-2)E(g(2n^2s))}{\int_Me^{ng(2n^2s)}\dd\mu}.
$$
These identities verify \eqref{eq:flow} and the volume constraint.
Since $|g(2n^2s)-g_0|\le1/8$, the displayed volume lies between
$e^{-n/8}$ and $e^{n/8}$, and
$e^{-1/4}\le\dd t/\dd s\le e^{1/4}$. For $c=e^{-1/4}\vartheta/(4n^2)$, the interval $[0,cr^4]$ is contained in the constructed normalized interval. The corresponding $\tau$ values on its terminal half lie in
$[e^{-1/2}T/4,T/2]$, where the preceding smoothing applies. The subtracted constant in \eqref{eq:normalize} is at most $1/8$ in absolute value. This proves the stated bounds. For smooth uniqueness, given a normalized smooth solution, let
$\zeta'(t)=-(n-2)E(f(t))$, $\zeta(0)=0$, and
$\tau'(t)=2n^2e^{2\zeta(t)}$. Then
$g(\tau(t))=f(t)+\zeta(t)$ satisfies the unnormalized equation. The change of time is invertible on each smooth finite interval, so unnormalized uniqueness gives normalized uniqueness.
\end{proof}

\subsection{Existence and smoothing from small \texorpdfstring{$C^0$}{C0} data}

\begin{proposition}\label{prop:small-local}
Let $(M^n,\omega)$ be a fixed compact connected smooth Hermitian manifold without boundary, of complex dimension $n\ge3$ and volume one. Assume either $R=0$, or that $\omega$ is Gauduchon and $R$ is constant. There exist $\varepsilon_0,h_0,\tau_0>0$, with $2h_0<1$, such that every real smooth $f_0$ satisfying the volume constraint and $\|f_0\|_\infty\le\varepsilon_0$ has a unique smooth solution of~\eqref{eq:flow} on $[0,h_0+\tau_0]$. For all integers $k\ge0$,
\begin{equation}\label{eq:small-smoothing}
\sup_{0\le t\le h_0+\tau_0}\|f(t)\|_\infty\le C\|f_0\|_\infty,
\qquad \|f(h_0)\|_{C^k}\le C_k\|f_0\|_\infty.
\end{equation}
The constants depend only on the fixed background and on $k$ where appropriate. The threshold $\varepsilon_0$ and the times are independent of all initial derivatives and of $k$.
\end{proposition}

\begin{proof}
We first omit the spatially constant term $(n-2)E(f)$ in~\eqref{eq:flow} and use the time $\tau$ equal to $2n^2$ times the unnormalized flow time. We denote its solution by $g$. When $R=0$, its equation is
\begin{equation}\label{eq:unnormalized-zero}
g_\tau=-e^{-ng}L^*(e^{(n-2)g}Lg)
-\frac{n-2}{2}e^{-2g}(Lg)^2.
\end{equation}
By~\eqref{eq:adjoint-identity}, its source terms relative to $L^*L$ are
$$
L^*((1-e^{-2g})Lg)-n\divmu(e^{-2g}Lg\nabla g)
+e^{-2g}\left[\frac{n+2}{2}(Lg)^2-\frac{n^2}{2}Lg|\nabla g|_\gamma^2\right].
$$
For $\|u\|_{X_T}\le1$, the product estimates proved above give
\begin{equation}\label{eq:nonlinear-bound}
\begin{split}
&\|(1-e^{-2u})Lu\|_{Y_{2,T}}+\|ne^{-2u}Lu\nabla u\|_{Y_{1,T}}\\
&\quad+\left\|e^{-2u}\left[\frac{n+2}{2}(Lu)^2-\frac{n^2}{2}Lu|\nabla u|_\gamma^2\right]\right\|_{Y_{0,T}}
\le C\|u\|_{X_T}^2.
\end{split}
\end{equation}
Differences have the bound $C(\|u\|_{X_T}+\|v\|_{X_T})\|u-v\|_{X_T}$, by factoring one polynomial or exponential difference at a time. These are smooth source maps on the fixed ball.

When the background is Gauduchon with constant $R$, expansion of the unnormalized equation adds to the right-hand side of~\eqref{eq:unnormalized-zero} the expression
$$
e^{-2g}\left[\frac{(n-2)R}{n}(L+L^*)g
+\frac{(n-2)^2R}{2n}|\nabla g|_\gamma^2-\frac{(n-2)R^2}{2n^2}\right].
$$
To check this, expand $R-nLg$ in the unnormalized equation and apply~\eqref{eq:adjoint-exponential} with $V=1$. The Gauduchon identity $L^*1=0$ eliminates its zeroth-order contribution.

Let $a(\tau)$ solve $a(0)=0$ and
$$
a_\tau=-\frac{(n-2)R^2}{2n^2}e^{-2a},\qquad
e^{2a(\tau)}=1-\frac{(n-2)R^2}{n^2}\tau.
$$
For the scalar-flat case, set $a=0$. Fix $0<T\le\min\{1,T_*\}$, reducing it in the constant-curvature case so that $(n-2)R^2T/n^2<1/20$. Write $g=a+u$. The initial value of $u$ is $f_0$, and $u=0$ is an exact solution with zero initial value. The right-hand side of the equation $u_\tau+L^*Lu$ has the following source representation:
\begin{align*}
&L^*\!\left(e^{-2a}(1-e^{-2u})Lu+(1-e^{-2a})Lu\right)
\\
&\quad-n\divmu\!\left(e^{-2(a+u)}Lu\,\nabla u\right)\\
&\quad+e^{-2(a+u)}\left[\frac{n+2}{2}(Lu)^2-\frac{n^2}{2}Lu\,|\nabla u|_\gamma^2\right]\\
&\quad+\frac{(n-2)R}{n}e^{-2(a+u)}(L+L^*)u
+\frac{(n-2)^2R}{2n}e^{-2(a+u)}|\nabla u|_\gamma^2\\
&\quad-\frac{(n-2)R^2}{2n^2}e^{-2a}(e^{-2u}-1).
\end{align*}
The extra highest-order source has coefficient $1-e^{-2a}=O(T)$. The second-order linear term has $Y_{0,T}$ norm bounded by $CT^{1/2}\|u\|_{X_T}$, by~\eqref{eq:lower}. The gradient-square term also has the factor $T^{1/2}$: on $B_\rho\times(0,\rho^4)$, use $Du\in L^4$ and then $L^2\subset L^1$ with the cylinder-volume factor; on $B_\rho\times(\rho^4/2,\rho^4)$, use $|Du|\le C\|u\|_{X_T}\rho^{-1}$. The last term has the factor $T$. Factoring differences gives the same estimates for the Lipschitz bound. All these constants may depend on the fixed $R$.

Lemma~\ref{lem:compact-linear} and~\eqref{eq:nonlinear-bound} now yield a contraction on a fixed ball in $X_T$. Choose $T$ so that all the extra linear contributions have operator norm less than $1/8$. Then choose the ball radius independently of $f_0$ so that the nonlinear Lipschitz contribution is less than $1/8$, and finally choose the initial $C^0$ threshold to place the homogeneous term inside this ball. The fixed point is unique there and satisfies
\begin{equation}\label{eq:X-small}
\|u\|_{X_T}\le C\|f_0\|_\infty.
\end{equation}

For smooth initial data, the unnormalized equation has a classical smooth short-time solution by~\cite[Theorem~1.1]{MM12}. The principal symbol is the negative of $\tfrac14e^{-2g}(\gamma^{ij}\xi_i\xi_j)^2$, and all lower coefficients are smooth functions of $g$ and derivatives through order three. The fourth-order tensor is the product of the two positive definite second-order factors, with the scalar factor $e^{-2g}$ absorbed into one of them, as required in that theorem. On a sufficiently short interval, the classical solution has small $X_T$ norm: all positive-derivative components of the norm tend to zero for fixed smooth data. It therefore coincides with the constructed solution by small-ball uniqueness.

We next prove positive-time smoothing with bounds independent of the initial derivatives. Fix a coordinate ball, choose smooth compactly supported vector fields $V_1,\ldots,V_d$ which are its coordinate fields on a smaller ball, and choose a smooth function $\theta$ of time equal to zero on $[0,T/4]$ and one on $[T/2,T]$. For $\zeta\in\R^d$ small, let $\Phi_{\zeta,\tau}$ be the time-$\theta(\tau)$ flow of $\sum_i\zeta_iV_i$. On a still smaller ball, $\Phi_{\zeta,\tau}(x)=x+\zeta$ for $\tau\ge T/2$.

Pullback by these maps is uniformly bounded on $X_T$. Cylinders with $\rho^4\le T/4$ are unchanged. For affected cylinders, $\rho\ge(T/4)^{1/4}$, and their images are covered by a fixed number of comparable cylinders when $\zeta$ is small. In the Hessian chain rule, the additional term $D^2\Phi\,Du$ has one lower derivative; its initial $L^2$ and terminal $L^q$ bounds have an extra factor $\rho$. The Jacobians and inverse Jacobians are uniformly bounded. These facts prove the pullback bound at both parts of~\eqref{eq:X}.

The function $v_\zeta=u\circ\Phi_{\zeta,\tau}$ satisfies the pulled-back equation in distributions and has the same initial value. Using the original operator $L^*L$, its equation differs from the fixed-point equation by a second divergence of a tensor coefficient times $D^2v$, where the coefficient is $O(|\zeta|)$ in $L^\infty$, and terms $D^j(c_\zeta D^kv)$ with $j+k\le3$. There is also a time-dependent first-order drift, bounded by $C_T|\zeta|$. Writing that drift as a divergence plus a zeroth-order term, and using~\eqref{eq:lower}, controls it in the source spaces. The nonlinear pulled-back terms retain the product orders just estimated. Their coefficients depend smoothly on $\zeta$. Thus the fixed-point equation for $v_\zeta$ remains a contraction in a common small ball after reducing $|\zeta|$. First fix the ball to include the uniform pullback bound, and reduce $\|f_0\|_\infty$ accordingly; this choice is independent of the derivative order to be estimated.

Let $v^{\mathrm{lin}}$ solve
$v_\tau^{\mathrm{lin}}+L^*Lv^{\mathrm{lin}}=0$ with
$v^{\mathrm{lin}}(0)=f_0$. Let $\mathcal N_\zeta$ be the map obtained
by applying the zero-initial-data solution operator in
Lemma~\ref{lem:compact-linear} to the source triple of the pulled-back
equation. It is smooth in $(\zeta,v)$, has $\mathcal N_\zeta(0)=0$, and has $\|D_v\mathcal N_\zeta\|<1/2$ on the common ball. The smooth implicit function theorem applied to $v-v^{\mathrm{lin}}-\mathcal N_\zeta(v)$ gives smooth dependence on $\zeta$. Differentiating once gives
$$
(I-D_v\mathcal N_\zeta(v_\zeta))\partial_{\zeta_i}v_\zeta
=\partial_{\zeta_i}\mathcal N_\zeta(v_\zeta).
$$
The inverse has norm at most two. Since $\partial_{\zeta_i}\mathcal N_\zeta(0)=0$, the right-hand side is bounded by $C_T\|v_\zeta\|_{X_T}$ and hence by $C_T\|f_0\|_\infty$. Each higher parameter derivative has the same invertible leading operator; every remaining term either vanishes at $v=0$ or contains a previously estimated derivative. Induction yields
$$
\|\partial_\zeta^\alpha v_\zeta\|_{X_T}\le C_{\alpha,T}\|f_0\|_\infty
$$
for all multi-indices $\alpha$ on one smaller parameter neighborhood. The actual distribution pullback lies in the common uniqueness ball, so it equals this fixed point. For $\tau>T/2$ on the inner chart, its parameter derivatives at zero are $D_x^\alpha u$ in distributions. The $L^\infty$ component of $X_T$, followed by a finite cover, proves
\begin{equation}\label{eq:interior-smoothing}
\sup_{T/2\le\tau\le3T/4}\|u(\tau)\|_{C^k}\le C_{k,T}\|f_0\|_\infty.
\end{equation}
Initially the estimates hold almost everywhere in spacetime. Every spatial derivative has a bounded distribution time derivative by the equation, since all spatial derivatives are now bounded. Spatial mollification and the fundamental theorem of calculus give compatible representatives continuous in space and locally Lipschitz in time. The equation then supplies a continuous time derivative; iteration supplies all mixed derivatives. Thus the solution is smooth on the positive-time interval and~\eqref{eq:interior-smoothing} holds at every displayed time. Applying this argument on smaller construction intervals proves smoothing at every positive time.

It remains to recover the exact volume-one flow and a common time. Define
\begin{align*}
t(\tau)&=\frac1{2n^2}\int_0^\tau
\left(\int_Me^{ng(\sigma)}d\mu\right)^{-2/n}d\sigma,\\
f(t(\tau))&=g(\tau)-\frac1n\log\int_Me^{ng(\tau)}d\mu.
\end{align*}
The unnormalized volume identity is
\begin{equation}\label{eq:unnormalized-volume}
\frac{d}{d\tau}\int_Me^{ng(\tau)}d\mu=-\frac{n-2}{2n}E(g(\tau)).
\end{equation}
It follows by integrating the adjoint term against $1$, using $L1=0$. The unnormalized vector field in the original time scales by $e^{-2c}$ under adding a constant $c$, while $E(g+c)=e^{(n-2)c}E(g)$. Since $d\tau/dt=2n^2(\int_Me^{ng}d\mu)^{2/n}$, differentiation of the displayed definition gives exactly~\eqref{eq:flow}. Its volume is one. This calculation applies in both background cases.

The constant reference $a$ cancels in the normalized factor:
$$
f=u-\frac1n\log\int_Me^{nu}d\mu.
$$
Since $\left|n^{-1}\log\int e^{nu}d\mu\right|\le\|u\|_\infty$, equation~\eqref{eq:X-small} gives the desired $C^0$ bound. The unnormalized volume is at most one by~\eqref{eq:unnormalized-volume}, and
$$
e^{n\min a-nC\|f_0\|_\infty}\le\int_Me^{ng}d\mu\le1.
$$
Reduce $\varepsilon_0$ so that $e^{2\min a-2C\varepsilon_0}>4/5$. The change of time then satisfies
$$
\frac45<\frac1{2n^2}\frac{d\tau}{dt}\le1.
$$
Set $h_0=5T/(16n^2)$ and $\tau_0=T/(16n^2)$. At $t=h_0$, the corresponding $\tau$ lies strictly between $T/2$ and $5T/8$. At $\tau=7T/8$, the normalized time is at least $7T/(16n^2)>h_0+\tau_0$. The uniform positive-time derivative estimates apply throughout the needed interior interval, giving~\eqref{eq:small-smoothing}. The chosen $T\le1$ gives $2h_0<1$.

Smooth uniqueness follows from the unnormalization argument in the proof of Proposition~\ref{prop:local}.
\end{proof}

\vspace{0.3cm}

{\small{\bf{Acknowledgments:}} {Part of this work was done during Ma's visit to Tianyuan mathematical center in southwest China in November of 2025. 
}

\vspace{0.5cm}


\textbf{Data Availability Statement}
Data sharing is not applicable to this article because no datasets were generated or analyzed during the current study.


\textbf{Conflict of interest} The authors declare that they have no conflict of interest.

\textbf{Code availability}
Not applicable.

\end{document}